\documentclass[11pt]{amsart}

\usepackage{epigamath}

\usepackage[english]{babel}

\numberwithin{equation}{section}

\usepackage{enumitem}
\usepackage{amsmath,amssymb,amsthm, tikz-cd}
\usepackage{graphicx}
\usepackage{soul}
\usepackage{dutchcal}

\usepackage{colonequals}

\usepackage{url}

\newtheorem{theorem}{Theorem}[section]
\newtheorem{corollary}[theorem]{Corollary}
\newtheorem{lemma}[theorem]{Lemma}
\newtheorem{proposition}[theorem]{Proposition}
\newtheorem{theorem-definition}[theorem]{Theorem/Definition}

\theoremstyle{definition}
\newtheorem{definition}[theorem]{Definition}
\newtheorem{setup}[theorem]{Setup}

\theoremstyle{remark}
\newtheorem{remark}[theorem]{Remark}
\newtheorem{example}[theorem]{Example}

\setlist[enumerate]{label = {\rm(\arabic*)}, ref={\rm\arabic*}}

\newlist{enuroman}{enumerate}{1}
\setlist[enuroman]{label = {\rm(\roman*)}, ref={\rm\roman*}}

\newlist{enualph}{enumerate}{1}
\setlist[enualph]{label = {\rm(\alph*)}, ref={\rm\alph*}}

\newcommand{\C}{\mathbb{C}}
\newcommand{\Z}{\mathbb{Z}}
\newcommand{\KK}{\mathbb{K}}

\newcommand{\cO}{\mathcal{O}}
\newcommand{\cC}{\mathcal{C}}

\newcommand{\cD}{\mathcal{D}}

\newcommand{\la}{\langle}
\newcommand{\ra}{\rangle}

\newcommand{\sq}{\subseteq}

\newcommand{\arrow}{arrow}
\newcommand{\F}{\mathbb{F}}
\newcommand{\Q}{\mathbb{Q}}
\newcommand{\N}{\mathbb{N}}
\newcommand{\A}{\mathbb{A}}

\newcommand{\fm}{\mathfrak{m}}
\newcommand{\fa}{\mathfrak{a}}
\newcommand{\fb}{\mathfrak{b}}

\newcommand{\spq}{\supseteq}

\newcommand{\Mustata}{Musta\c{t}\u{a}}
\newcommand{\ceq}{\colonequals}

\DeclareMathOperator{\Hom}{Hom}
\DeclareMathOperator{\End}{End}

\DeclareMathOperator{\Spec}{Spec}

\DeclareMathOperator{\id}{id}

\DeclareMathOperator{\im}{im}
\DeclareMathOperator{\Ann}{Ann}

\DeclareMathOperator{\Fun}{Fun}
\DeclareMathOperator{\FJN}{FJN}

\DeclareMathOperator{\mmod}{-mod}

\DeclareMathOperator{\str}{str}
\DeclareMathOperator{\Gal}{Gal}

\DeclareMathOperator{\BSR}{BSR}
\newcommand{\bs}[1]{\boldsymbol{#1}}
\newcommand{\lra}{\longrightarrow}

\def\lowsim{\vbox to 0pt{\vss\hbox{$\scriptstyle\sim$}\vskip-1.6pt}}
\def\highsim{\vbox to 0pt{\vss\hbox{$\scriptstyle\sim$}\vskip2.5pt}}

\EpigaVolumeYear{9}{2025} \EpigaArticleNr{18} \ReceivedOn{November 13, 2024}
\InFinalFormOn{March 4, 2025}
\AcceptedOn{March 25, 2025}

\title{Bernstein--Sato theory modulo $p^m$}
\titlemark{Bernstein--Sato theory modulo $p^m$}

\author{Thomas Bitoun}
\address{Shanghai Institute for Mathematics and Interdisciplinary Sciences 
Block A, International Innovation Plaza, No.~657 Songhu Road, Yangpu District, Shanghai, China}
\email{thomas@simis.cn}

\author{Eamon Quinlan-Gallego}
\address{University of Illinois at Chicago, 
851 S. Morgan Street, 322 Science and Engineering Offices, Chicago, IL, USA}
\email{eamonqg@uic.edu}

\authormark{T.~Bitoun and E.~Quinlan-Gallego}

\AbstractInEnglish{For fixed prime integer $p > 0$, we develop a notion of Bernstein--Sato polynomial for polynomials with $\Z / p^m$-coefficients, compatible with existing theory in the case $m = 1$. We show that the ``roots'' of such polynomials are rational, and we show that the negative roots agree with those of the mod-$p$ reduction. We give examples to show that, surprisingly, roots may be positive in this context. Moreover, our construction allows us to define a notion of ``strength'' for roots by measuring $p$-torsion, and we show that ``strong'' roots give rise to roots in characteristic zero through mod-$p$ reduction.}


\acknowledgement{Eamon Quinlan-Gallego was supported by the National Science Foundation under award numbers 1801697, 1840190, 1840234, and 2203065. This work was supported by the Natural Sciences and Engineering Research Council of Canada (NSERC), [RGPIN-2020-06075]. Cette recherche a \'et\'e financ\'ee par le Conseil de recherches en sciences naturelles et en g\'enie du Canada (CRSNG), [RGPIN-2020-06075].}

\begin{document}



\maketitle

\begin{prelims}

\DisplayAbstractInEnglish

\bigskip

\DisplayKeyWords

\medskip

\DisplayMSCclass

\end{prelims}


\newpage

\setcounter{tocdepth}{1}

\tableofcontents


\section{Introduction} \label{scn-intro}

Let $R := \C[x_1, \dots , x_n]$ be a polynomial ring over $\C$, and let $\cD_R := \C \la x_i , \frac{\partial}{\partial x_i} \ra$ be the Weyl algebra. Given a nonzero $f \in R$, its Bernstein--Sato polynomial $b_f(s) \in \C[s]$ is the monic polynomial of smallest degree satisfying a functional equation of the form
$$b_f(s) f^s = P(s) f^{s+1}$$
for some $P(s) \in \cD_R[s]$. Such a $b_f(s)$ always exists, and its roots are rational and negative by a theorem of Kashiwara \cite{Kas76}. 

This Bernstein--Sato polynomial detects important information about the singularities of $f$. For example, Kashiwara and Malgrange showed that one can recover the eigenvalues of the monodromy action on the nearby cycles from the roots of $b_f(s)$; see \cite{Kas83,Mal83}.  Koll\'ar also showed that the log-canonical threshold of $f$ is the negative of the largest root of the Bernstein--Sato polynomial, see \cite{Kollar97}, and afterwards Ein, Lazarsfeld, Smith, and Varolin proved that whenever $\alpha \in (0,1]$ is a jumping number for the multiplier ideals of $f$, this $\alpha$ must be a root of $b_f(-s)$; see \cite{ELSV04}.

  In the hope of building new invariants of singularities, we develop a notion of ``Bernstein--Sato polynomial'' in the case where one replaces the base $\C$ by $\Z / p^{m+1}$, where $p$ is a prime number (the theory works over a more general base, see Setup~\ref{setup-V}, but we focus on this case here for simplicity). This can be seen as a necessary step in developing a genuinely $p$-adic theory. Note that the case $m = 0$ (\textit{i.e.}~when one works over~$\F_p$) has already been treated in work of \Mustata \ \cite{Mustata2009} and the first author \cite{Bitoun2018} (see also \cite{BS16,QG19,QG19b,JNBQG} for more references).

The replacement of $\C$ with the base $\Z / p^{m+1}$ in the definition of $b_f(s)$ has two major components, which we illustrate with the following questions:  

\begin{enumerate}
	\item\label{item1} What should play the role of the algebra $\C \la x_i, \frac{\partial}{\partial x_i} \ra$?
	\item\label{item2} How does one give a ``functional equation'' in this setting?
\end{enumerate}

Let us briefly describe how these are dealt with.

We address question~\eqref{item1} by considering the ring of differential operators as defined by Grothendieck \cite[Chapter~16]{EGAIV}, which assigns to every commutative algebra $R$ over a commutative ring $A$ a subring $\cD_{R|A} \sq \End_A(R)$ of the ring $\End_A(R)$ of $A$-linear endomorphisms of $R$. We will write this ring as $\cD_R$ whenever $A$ is clear from the context, and we note that it is noncommutative in general.

We refer the reader to Section~\ref{subscn-background-diffops} for the precise definition. For now, let us mention that, in the case where $R := A [x_1, \dots , x_n]$ is a polynomial ring over $A$, one can give algebra generators for $\cD_R$ as follows:
$$\cD_R = A \left\la x_i, \partial_i^{[k]} \ \middle\vert \ i = 1, \dots , n ; \ k = 1, 2, \cdots \right\ra, $$
where through an abuse of notation one thinks of $x_i \in \End_A(R)$ as the operator that multiplies by the variable $x_i$, and $\partial_i^{[k]} \in \End_A(R)$ is the operator determined by
\begin{align*}
	\partial_i^{[k]} \left(x_1^{a_1} \cdots x_n^{a_n}\right) & := \binom{a_i}{k} x_i^{-k} x_1^{a_1} \cdots x_n^{a_n} \\
		& = \frac{a_i (a_i - 1) (a_i - 2) \cdots (a_i - k + 1)}{k!} x_i^{-k} x_1^{a_1} \cdots x_n^{a_n}
\end{align*}	
(recall that $\binom{a_i}{k}$ is always an integer, and hence this is defined regardless of whether one can divide by~$k!$ in $A$; also note  that $\binom{a_i}{k} = 0$ for $a_i < k$ and hence the above is indeed an element of $R$). Moreover, when $A$ is $\C$ (or any field of characteristic zero), we have $\partial_i^{[k]} = \partial_i^k / k!$, and one recovers $\cD_R =  \C \la x_i, \partial_i \ra$. Finally, let us mention that whenever $A = \Z / p^{m+1}$, we give an alternative characterization of this algebra in Proposition~\ref{prop-diffops-Frobenius}. 

For question~\eqref{item2} we use an alternative description of the Bernstein--Sato polynomial that translates naturally to an arbitrary base $A$. Basically, if $R := A[x_1, \dots , x_n]$ is a polynomial ring over $A$ and $f \in R$ is a nonzerodivisor, one considers the polynomial ring $R[t]$ in one extra variable, to which one gives an $\N$-grading by declaring $\deg x_i = 0$ and $\deg t = 1$; this induces a $\Z$-grading on $\cD_{R[t]}$. One then assigns to $f$ a module $N_f$ over the ring $(\cD_{R[t]})_0$ of differential operators of degree zero (we refer the reader to Section~\ref{scn-b-function-Zpm} for a detailed description of $N_f$). In general, every differential operator on $R$ induces a differential operator on $R[t]$, giving an inclusion $\cD_R \sq \cD_{R[t]}$. 

Let us describe what happens in the classical case $A = \C$. Here we get $(\cD_{R[t]})_0 = \cD_R[s]$, where $s := - \partial_t t$, and the Bernstein--Sato polynomial is then the minimal polynomial for the action of $s$ on $N_f$. Note that $N_f$ is in general infinite dimensional over $\C$, and hence the fact that $b_f(s) \neq 0$ is far from trivial, but nonetheless true. From this one concludes that $N_f$ splits as a direct sum $N_f = \bigoplus_{\lambda \in \C} (N_f)_\lambda$ of generalized eigenspaces for the action of $s$, and moreover one can recover the roots of $b_f(s)$ from this decomposition by
$$\left\{ \text{Roots of $b_f(s)$} \right\} = \left\{ \lambda \in \C \ | \ \left(N_f\right)_\lambda \neq 0 \right\}.$$
One can also think of this fact geometrically: if one thinks of the $\C[s]$-module $N_f$ as a sheaf on $\Spec(\C[s]) = \A^1_\C$, then it is supported at a finite number of points, and these points are precisely the roots of $b_f(s)$. As we will see, when working over $\Z / p^m$, one thinks of this finite-support property as an analogue of the ``existence theorem'' for Bernstein--Sato polynomials. 

As mentioned, the construction of the $(\cD_{R[t]})_0$-module $N_f$ can be carried out over an arbitrary base $A$, and in particular we can consider the case  $A = \Z / p^{m+1}$. The most striking difference in this situation is that $(\cD_{R[t]})_0$ has nonredundant higher-order differential operators, 
$$\left(\cD_{R[t]}\right)_0 = \cD_R\left[- \partial_t t, - \partial^{[p]} t^p, - \partial^{[p^2]} t^{p^2}, \dots \right],$$
and in particular the role of $\C[s]$ is now played by the algebra
$$C_m := \left(\Z / p^{m+1}\right) \left[- \partial_t t, - \partial^{[p]} t^p, - \partial^{[p^2]} t^{p^2}, \dots \right].$$

We refer the reader to Section~\ref{subscn-alg-cts-function} for more usable characterizations of this algebra. For now, let us just mention the following fact (note that in the case $m = 0$ these facts already appear in \cite{Bitoun2018,JNBQG}).

\begin{proposition}[see Section 2.3]
	The algebra $C_m$ is commutative, all of its prime ideals are maximal, and $\Spec(C_m)$ is homeomorphic to the space $\Z_p$ of $p$-adic integers equipped with the $p$-adic topology.
\end{proposition}

In particular, to each $p$-adic integer $\alpha \in \Z_p$, one can assign a maximal ideal $\fm_\alpha \sq C_m$, and every prime ideal of $C_m$ arises this way. From now on let us identify $\Spec(C_m) = \Z_p$. Hence we think of $N_f$ as a sheaf on $\Z_p$. 

Let us summarize what is known about the case $m = 0$. The following theorem of the first author can be seen as an analogue of the existence of Bernstein--Sato polynomials, as an analogue of Kashiwara's theorem of rationality and negativity of their roots, and as an analogue of the theorem of Ein--Lazarsfeld--Smith--Varolin that connects these roots to jumping numbers. 

\begin{theorem-definition} [\textit{cf.} \cite{Bitoun2018}] \label{thm-bitoun-intro}
	Suppose $f \in \F_p[x_1, \dots , x_n]$ is a nonzero polynomial. The module $N_f$ is supported at a finite number of $p$-adic integers $\alpha_1, \dots , \alpha_s \in \Z_p$. These $p$-adic integers are called the Bernstein--Sato roots of $f$, and we denote this collection by $\BSR(f) \sq \Z_p$. 

	One has
	$$\BSR(f) = - \left( \FJN(f) \cap \Z_{(p)} \cap (0,1] \right),$$
	where $\FJN(f)$ denotes the collection of\, $F$-jumping numbers of $f$. In particular, the Bernstein--Sato roots of $f$ are rational, negative, and lie in the interval $[-1, 0)$. 
\end{theorem-definition}

This result raises the question of whether one can develop a theory that also includes roots smaller than $-1$, and whether one can assign a multiplicity to each root (in characteristic zero, the roots that are smaller than $-1$ carry important information about the singularities of $f$; see for example \cite{Saito-bsrs,MP20}).
	
Let us now talk about our results; in summary, we show that passing to arbitrary $m$ does not produce new negative Bernstein--Sato roots, but that (surprisingly) Bernstein--Sato roots can be positive in this setting. We also develop a multiplicity-like invariant that, despite not providing a good analogue for multiplicity, does give interesting information.

We begin by proving the finite-support statement for the module $N_f$, which we use to define Bernstein--Sato roots.

\begin{theorem-definition}[\textit{cf.} Definition~\ref{def-BSR} and Theorem~\ref{thm-BSR-fin}]
	Let $f \in (\Z / p^{m+1})[x_1, \dots , x_n]$ be a nonzerodivisor. Then the $C_m$-module $N_f$ is supported at a finite number of $p$-adic integers $\alpha_1, \dots , \alpha_s \in \Z_p$. These $p$-adic integers are called the Bernstein--Sato roots of $f$, and we denote this collection by $\BSR(f) \sq \Z_p$. 
\end{theorem-definition}

We are able to show that these Bernstein--Sato roots are rational. From existing theory in characteristic zero and characteristic $p > 0$ ($m = 0$), one would expect them to be negative. We give examples to show that this need not be the case (see Example~\ref{ex-pos-BSR}), but we prove that the roots that {\em are} negative agree with the Bernstein--Sato roots of the mod-$p$ reduction of $f$. We also show that every positive root is an integer translate of a negative one. In summary, we have the following. 

\begin{theorem}[\textit{cf.} Theorems~\ref{thm-BSR-rat},~\ref{thm-BSR-neg}, and~\ref{thm-Z-translate}]
	Let $f \in (\Z / p^{m+1})[x_1, \dots , x_n]$ be a nonzerodivisor, and let $f_0 \in \F_p[x_1, \dots , x_n]$ be its mod-$p$ reduction.
	\begin{enuroman}
	\item The Bernstein--Sato roots of $f$ are rational; \textit{i.e.}~$\BSR(f) \sq \Z_{(p)}$.
	\item We have $\BSR(f) \cap (\Z_{(p)})_{< 0} = \BSR(f_0)$.
	\item We have $\BSR(f) + \Z = \BSR(f_0) + \Z$.
	\end{enuroman}
\end{theorem}

Given a $p$-adic integer $\alpha \in \Z_p$, the powers of the maximal ideals $\fm_\alpha \sq C_m$ form a decreasing chain, which takes the form 
$$\fm_\alpha \supseteq \fm_\alpha^2 \supseteq \fm_\alpha^3 \supseteq \cdots \supseteq \fm_\alpha^{m+1} = \fm_\alpha^{m+2} = \fm_\alpha^{m+3} = \cdots$$
(in particular, when $m = 0$, the ideals $\fm_\alpha$ are idempotent; that is, $\fm_\alpha = \fm_\alpha^2$). Given a Bernstein--Sato root~$\alpha$ for $f$, we define the number
$$\str(\alpha, f) := \min\left\{k \geq 0 \ \middle\vert \ \fm_\alpha^k \left(N_f\right)_\alpha = 0 \right\},$$
where $(N_f)_\alpha$ is the stalk of $N_f$ at $\alpha$ (note that in characteristic zero this gives the multiplicity of a root~$\alpha$). Some quick computations show that this does not give an analogue for multiplicity; for example, for $f = x_1$ one has $\str(-1, x_1) = m + 1$ (see Example~\ref{ex-x1-strength}). For this reason we give this invariant another name, and we call $\str(\alpha, f)$ the {\it strength} of $\alpha$ in $f$; this is a new feature of the theory that is not available when $m = 0$. We show that these invariants detect whether $\alpha$ arises as a root of the Bernstein--Sato polynomial in characteristic zero.

\begin{theorem}[\textit{cf.} Corollary~\ref{cor-strength-b}]
Let $f \in \Z[x_1, \dots , x_n]$ be a polynomial with integer coefficients, and let $p$ be a large enough prime. For every integer $m \geq 0$, let $f_m \in (\Z / p^{m+1})[x_1, \dots , x_n]$ be the mod-$p^{m+1}$ reduction of $f$. Given $\alpha \in \Z_{(p)}$, the sequence $\left( \str(\alpha, f_m) \right)_{m = 0}^\infty$ is nondecreasing and satisfies 
$$\nu_p \left( b_f(\alpha) \right) \geq \str\left(\alpha, f_m\right)$$
for every $m \geq 0$, where $\nu_p(-)$ denotes $p$-adic valuation. In particular, if\, $(\str(\alpha), f_m)_{m = 0}^\infty$ is unbounded, then we have $b_f(\alpha) = 0$.
\end{theorem}

\subsection*{Acknowledgments}

We would like to thank Christopher Dodd, Jack Jeffries, Kiran Kedlaya, Pedro L\'opez Sancha, Mircea \Mustata, Karl Schwede, Kevin Tucker, and Jakub Witaszek for useful comments and conversations throughout the course of this project.

\section{Background, preliminaries, and notation}

In this paper we will always work over a very specific class of base rings; we begin by formalizing this in the following.

\begin{setup} \label{setup-V}
  Let $(V, \fm, \KK)$ be a commutative artinian local ring with residue field $\KK$ of characteristic $p > 0$. We assume that there exists a bijective ring homomorphism $F\colon V \to V$ such that the induced map $\KK \to \KK$ is the Frobenius morphism. We fix such an $F$, which will be referred to as a lift of Frobenius.
\end{setup}

\begin{remark} \label{rmk-V}
	Let $(V, \fm, \KK, F)$ be as in Setup~\ref{setup-V}. Since $F\colon V \to V$ is bijective by assumption, it must be a bijection on nonunits, and therefore $F(\fm) = \fm$. The bijectivity of $F$ also implies that the Frobenius on $\KK$ is surjective, and hence $\KK$ is forced to be perfect.
\end{remark}

\begin{example}
	We provide some examples of the above setup. 
	\begin{enuroman}
	\item The simplest example is the following: $V = \Z / p^{m+1}$, $\fa = (p)$, and $F\colon V \to V$ being the identity map.
	\item More generally, if $k$ is a perfect field of characteristic $p > 0$, one could take $V = W_{m+1}(k)$ to be the ring of Witt vectors of length $m + 1$ over $k$, with $\fm = (p)$ and $F$ the Frobenius endomorphism on Witt vectors.
	\item If $K$ is a finite Galois extension of $\Q_p$ with ring of integers $(\cO_K, \fm, k)$, one can take $V = \cO_K / \fm^{m+1}$. It is well known that the induced map $\Gal(K / \Q_p) \to \Gal(k / \F_p)$ is surjective, see \cite[Section~7, Proposition~20]{Serre-local-fields}, so there exists a lift of Frobenius.
	\item There are interesting examples in equal characteristic $p$ which also satisfy the conditions of Setup~\ref{setup-V}; for example, one could take $V = \F_p[t] / (t^{m+1})$ with $F(t) = t$. This gives rise to some interesting questions, which we will not tackle.
	\end{enuroman}
\end{example}

\subsection{Basics on lifts of Frobenius} \label{subscn-basics-lof}

Let $(V, \fm, \KK, F)$ be as in Setup~\ref{setup-V}. Given a $V$-algebra $R$, we define what it means for a lift of Frobenius $R \to R$ to be compatible with the lift of Frobenius $F\colon V \to V$. 

\begin{definition}
	Let $(V, \fm, \KK, F)$ be as in Setup~\ref{setup-V}, and let $R$ be a $V$-algebra. A ring endomorphism $F_R\colon R \to R$ is called a compatible lift of Frobenius if the diagram
	$$\begin{tikzcd}
		V \arrow[r, "F"] \arrow[d] & V \arrow[d] \\
		R \arrow[r, "F_R", swap] & R
	\end{tikzcd}$$
	commutes and the induced map $R / \fm R \to R / \fm R$ is the Frobenius endomorphism. If no confusion is likely to arise, we may also write $F$ for such a lift of Frobenius on $R$. 
\end{definition}

\begin{remark}
	When $R$ is smooth over $V$, we can always fill in the dashed arrow in the commutative diagram
	$$\begin{tikzcd}
		V \arrow[r] & R \arrow[rrd, twoheadrightarrow] & \ & \ \\
		V \arrow[u, "F"] \arrow[r] & R \arrow[u, dashed] \arrow[r] & R / \fm R \arrow[r, "\text{Frob}", swap] & R / \fm R\rlap{,}
	\end{tikzcd}$$
	which tells us that there exists a compatible lift of Frobenius. 
\end{remark}

\begin{example}
	If $R \ceq V[x_1, \dots , x_n]$, then $R$ admits a compatible lift of Frobenius $F\colon R \to R$ given by $F_R(a) = F(a)$ for all $a \in V$ and $F_R(x_i) = x_i^p$ for all $i = 1, \dots , n$. More generally,  one can instead declare $F_R(x_i) = x_i^p + y_i$, where $y_1, \dots , y_n$ are arbitrary elements of $\fm R$. 
\end{example}

Let $R$ be a $V$-algebra and $F\colon R \to R$ be a lift of Frobenius on $R$. Given an integer $e \geq 0$, we let $F^e\colon R \to R$ be the ring endomorphism given by $F^e \ceq F \circ \cdots \circ F$ ($e$ times), with the understanding that $F^0$ is the identity on $R$. 

Recall that, given a homomorphism $\phi\colon A \to B$ of commutative rings, one obtains an associated restriction of scalars functor $\phi_*\colon B\mmod \to A\mmod$, as well as an extension of scalars functor $\phi^*\colon A \mmod \to B\mmod$. The same is true, therefore, for the endomorphisms $F^e\colon R \to R$, but the fact that these endomorphisms have the same source and target makes these functors slightly confusing. We introduce some standard notation to try to alleviate this confusion. 

If $M$ is an $R$ module and $e \geq 0$ is an integer, we let $F^e_* M$ be the $R$-module obtained through restriction of scalars along $F^e\colon R \to R$. An element $u \in M$ will be denoted by $F^e_* u$ when we want to emphasize that we view $u$ as an element of $F^e_* M$ instead of $M$. With this notation, the $R$-module structure of $F^e_* M$ is given by the expression 
$$g \cdot F^e_* u = F^e_* \left( F^e(g) \cdot u \right),\quad \text{ where } g \in R, u \in M.$$

This construction is rigged so that $F^e\colon R \to F^e_* R$ is $R$-linear; in particular, we may think of $F^e_* R$ as an $R$-algebra, where the $R$-module structure is given through $F^e$. Note that, given integers $i, j \geq 0$, the set $\Hom_R(F^i_* R, F^j_* R)$ is given by
\begin{equation} \label{eqn-hom}
\Hom_R\left(F^i_* R, F^j_* R\right) = \left\{ \phi \in \Hom_\Z(R, R) \ \middle\vert \ \phi(F^i(g) h) = F^j(g) \phi(h) \text{ for all } g, h \in R \right\}.
\end{equation}

Given an $R$-module $M$, the extension of scalars along $F^e$ is given by $F^{e*}(M) = B \otimes_R M$, where $B$ is the $(R, R)$-bimodule given by $B = R$ as a left $R$-module and $B = F^e_* R$ as a right $R$-module. For future purposes, it will be useful to think of $B$ as $B = \Hom_R(R, F^e_* R)$, where the bimodule structure on $\Hom_R(R, F^e_* R)$ is given by
\begin{align*}
	\phi \cdot g = \phi \circ g, \\
	g \cdot \phi = F^e_* g \circ \phi
\end{align*}
for all $\phi \in \Hom_R(R, F^e_* R)$ and $g \in R$. With this notation, therefore, we have
$$F^{e*}(M) = \Hom_R(R, F^e_* R) \otimes_R M.$$

\begin{example} \leavevmode
	\begin{enuroman}
	\item One has $F^{e*}(R) \cong R$. 
	\item Suppose an $R$-module $M$ is presented as 
	$$
		R^{\oplus m} \overset{A}\lra R^{\oplus n} \lra M \lra 0,
	$$
	where $A = (a_{ij})$ is an $n \times m$ matrix over $R$. Then the module $F^{e*}(M)$ is presented by 
	$$
		R^{\oplus m} \xrightarrow{F^e(A)} R^{\oplus n} \lra F^{e*}(M) \lra 0,
	$$
	where the matrix $F^e(A)$ is given by $F^e(A) = (F^e(a_{ij}))$. 
	\end{enuroman}
\end{example}

Recall that given a homomorphism $A \to B$ of commutative rings and an ideal $I \sq A$, one obtains the ideal $IB \sq B$ as the extension of $I$. Once again, the same is the case for the endomorphisms $F^e\colon R \to R$, but the fact that the source and target are the same ring means that more careful notation is required. 

\begin{definition} \label{def-expansion-ideal-Frobenius}
	Let $(V, \fm, \KK, F)$ be as in Setup~\ref{setup-V}, let $R$ be a $V$-algebra, and let $F\colon R \to R$ be a compatible lift of Frobenius. Given an ideal $J \sq R$ and an integer $e \geq 0$, we define the new ideal $\left( F^e(J) \right) \sq R$ as the extension of $I$ along $F^e\colon R \to R$. In other words, we have 
	$$ \left( F^e(J) \right) \ceq \left(F^e(g) \ | \ g \in J\right).$$
\end{definition}

\begin{remark}
	To compute $\left( F^{e} (J) \right)$ it suffices to apply $F^e$ to the generators of $J$. More precisely, given generators $J = (g_i \ | \ i \in I )$, we have $ \left( F^{e} (J) \right) = \left(F^e(g_i) \ | \ i \in I\right)$.
\end{remark}

\begin{remark} \label{rmk-ideal-F-pullback}
	Applying the functor $F^{e*}(-)$ to the inclusion $J \sq R$ yields a map
	$$F^{e*}(J) \lra F^{e*}(R) \cong R,$$
	whose image is precisely $\left( F^e(J) \right)$. We will later see that, when $R$ is smooth over $V$, the map $F^e\colon R \to R$ is flat (see Lemma~\ref{lemma-FeR-fgp}), and therefore one can identify $\left( F^e (J) \right) \cong F^{e*}(J)$.
\end{remark}

\begin{lemma} \label{lemma-p-power}
	Let $(V, \fm, \KK, F)$ be as in Setup~{\rm\ref{setup-V}}, and let $R$ be a $V$-algebra. If $g, h \in R$ are elements such that $g \equiv h \mod \fm^k R$ for some integer $k \geq 1$, then $g^{p^e} \equiv h^{p^e} \mod \fm^{k+e} R$ for all $e \geq 0$. 
\end{lemma}
\begin{proof}
	By induction it suffices to prove the claim in the case $e = 1$. Let $x \in \fm^k R$ be such that $g = h + x$. The binomial expansion then gives
	\begin{align*}
		g^p & = h^p + \sum_{i = 1}^{p} \binom{p}{i}\, x^i g^{p-i}.  
	\end{align*}
	To conclude the proof, we claim that $\binom{p}{i} x^i \in \fm^{k+1} R$ for all $i = 1, \dots, p$. 

	To prove the claim, first recall that in $V$ we have $p \in \fm$. In the case $i \neq p$, the claim then follows because $p$ divides $\binom{p}{i}$, and hence $\binom{p}{i} x^i \in \fm^{1 + ik} R \sq \fm^{k + 1} R$. In the case $i = p$, recall that $k \geq 1$ by assumption and therefore $x^p \in \fm^{kp} \sq \fm^{k + 1}$.
\end{proof}

\begin{lemma} \label{lemma-Frob-p-power}
	Let $(V, \fm, \KK, F)$ be as in Setup~{\rm\ref{setup-V}}, and fix an $m \geq 0$ such that $\fm^{m+1} = 0$ in $V$. Let $R$ be a $V$-algebra equipped with a compatible lift of Frobenius $F\colon R \to R$. For all integers $e \geq 0$ and all $g \in R$, we have
	$$F^e\left(g^{p^m}\right) = g^{p^{m+e}}.$$
\end{lemma}
\begin{proof}
	Since $F$ lifts Frobenius on $R / \fm R$, we know that $F^e(g) \equiv g^{p^e} \mod \fm R$. By Lemma~\ref{lemma-p-power} we conclude that $F^e(g^{p^m}) = F^e(g)^{p^m} \equiv g^{p^{e+m}} \mod \fm^{m+1} R$. Since $\fm^{m+1} = 0$ by assumption, the statement follows.
\end{proof}

\begin{proposition} \label{prop-Frob-lift-localization}
	Let $(V, \fm, \KK, F)$ be as in Setup~{\rm\ref{setup-V}}, let $R$ be a $V$-algebra with a compatible lift of Frobenius $F\colon R \to R$, and let $W \sq R$ be a multiplicative subset. There is a unique extension of\, $F\colon R \to R$ to a compatible lift of Frobenius on $W^{-1} R$. 
\end{proposition}
\begin{proof}
	It suffices to show that, for every $w \in W$, the element $F(w)$ is invertible in $W^{-1} R$. This follows because $F(w) \equiv w^p \mod \fm W^{-1} R$ and $\fm W^{-1} R$ is nilpotent; alternatively, if $m \geq 0$ is such that $\fm^{m+1} = 0$ in $V$, then by Lemma~\ref{lemma-Frob-p-power} we have $F(w)^{p^m} = w^{p^{m+1}}$. 
\end{proof}

\subsection{Rings of differential operators and Frobenius descent} \label{subscn-background-diffops}

Let $V$ be a commutative ring and $R$ be a commutative $V$-algebra. A construction of Grothendieck assigns to this data a subring $\cD_{R|A}$ of the (noncommutative) ring $\End_A(R)$ of $A$-linear endomorphisms of $R$; see \cite[Chapter~16]{EGAIV}. Here we briefly explain this construction and provide an alternative characterization of this ring in the case where the base ring $V$ is as in Setup~\ref{setup-V}. Note that some of these results have already been obtained by Berthelot \cite{Berthelot-divided}, but in our exposition we will avoid the use of divided powers.

We begin by noting that $\End_V(R)$ has a natural $R \otimes_V R$-module structure, where a pure tensor $r \otimes s \in R \otimes_V R$ acts on an element $\phi \in \End_V(R)$ by
$$\left( (r \otimes s) \cdot \phi\right)(x) := r \phi(sx)$$
for all $x \in R$. We let $J_{R|V} \sq R \otimes_V R$ denote the kernel of the natural map $R \otimes_V R \to R$ given by multiplication. The ideal $J_{R|V}$ defines the diagonal subscheme $\Spec(R) \hookrightarrow \Spec(R) \times_V \Spec(R)$, and we have generators $J_{R|V} = (1 \otimes r - r \otimes 1 \ | \ r \in R )$. 

\begin{definition}[\textit{cf.} {\cite[Chapter~16]{EGAIV}}]
	Let $R$ be a $V$-algebra. The ring 
	$$\cD_{R|V} \ceq \left\{ \phi \in \End_V(R) \ \middle\vert \ J_{R|V}^n \cdot \phi = 0 \text{ for some } n \gg 0 \right\}$$
	is called the ring of $V$-linear differential operators on $R$.
\end{definition}

We will write this as $\cD_R$ whenever $V$ and the $V$-algebra structure of $R$ are clear from the context. It is clear that $\cD_{R|V}$ is closed under addition; one can verify that it is also closed under composition, and thus $\cD_{R|V}$ forms a subring of $\End_V(R)$.

Now suppose that $V$ is as in Setup~\ref{setup-V}. We show that whenever a $V$-algebra $R$ admits a lift of Frobenius, one can describe $\cD_{R|V}$ in terms of it. Let us begin by introducing the relevant notation.

\begin{definition}
	Let $(V, \fm, \KK, F)$ be as in Setup~\ref{setup-V}, and let $R$ be a $V$-algebra. Let $F\colon R \to R$ be a compatible lift of Frobenius and $e \geq 0$ be an integer. Then the ring
	$$\cD^{(F, e)}_R \ceq \Hom_{F^e(R)} (R, R)$$
	is called the ring of differential operators of level $e$ with respect to $F$.
\end{definition}

Note that one can also think of $\cD^{(F, e)}_R$ through the following description (see Equation (\ref{eqn-hom})):
$$\cD^{(F, e)}_R = \Hom_R\left(F^e_* R, F^e_* R\right).$$
Moreover, the lift of Frobenius $F\colon R \to R$ induces a compatible lift of Frobenius $F \otimes F\colon R \otimes_V R \to R \otimes_V R$. One easily verifies that $(F \otimes F)(J_{R|V}) \sq J_{R | V}$, and therefore we get a descending chain of ideals
$$J_{R|V} \supseteq \left( (F \otimes F) \left(J_{R|V}\right) \right) \supseteq \left( (F \otimes F)^{2} \left(J_{R|V}\right) \right) \supseteq \cdots.  $$
We observe that $\cD^{(F, e)}_R$ can be described by using these ideals as follows.

\begin{lemma} \label{lemma-diffops-level-ann}
	Let $(V, \fm, \KK, F)$ be as in Setup~{\rm\ref{setup-V}}, let $R$ be a $V$-algebra, and let $F\colon R \to R$ be a compatible lift of Frobenius. For every integer $e \geq 0$, we have
	$$\cD^{(F, e)}_R = \left\{ \phi \in \End_V(R) \ | \  \left((F \otimes F)^{e}\left(J_{R|V}\right) \right) \cdot \phi = 0 \right\}.$$
\end{lemma}
\begin{proof}
	First note that, by our assumption in Setup~\ref{setup-V}, the morphism $F$ is bijective on $V$, and therefore $V \sq F^e(R)$ for every $e \geq 0$, which shows that $\cD^{(F, e)}_R \sq \End_V(R)$.

	The ideal $J_{R|V}$ is generated by elements of the form $1 \otimes r - r \otimes 1$ for $r \in R$. It follows that $(F \otimes F)^{e}(J)$ is the ideal generated by elements of the form $1 \otimes F^e(r) - F^e(r) \otimes 1$, where $r$ ranges through all elements of $R$. Observe that $(1 \otimes F^e(r) - F^e(r) \otimes 1) \cdot \phi = 0$ if and only if $\phi$ commutes with multiplication by $F^e(r)$.	
\end{proof}

\begin{lemma} \label{lemma-cofinal}
	Let $(V, \fm, \KK, F)$ be as in Setup~{\rm\ref{setup-V}}. Let $S$ be a flat $V$-algebra, $F\colon S \to S$ be a compatible lift of Frobenius, and $J \sq S$ be an ideal. Assume that $S / J^k S$ is flat over $V$ for every integer $k \geq 1$, that $J$ is finitely generated, and that $F(J) \sq J$. Then the families of ideals $\{J^n\}_{n = 0}^\infty$ and $\{( F^{e} (J) ) \}_{e = 0}^\infty$ are cofinal.
\end{lemma}
\begin{proof}
	Let $m \geq 0$ be such that $\fm^{m+1} = 0$ in $V$, fix generators $J = (f_1, \dots , f_s)$ for $J$, and fix an integer $e \geq 0$. The ideal $J^{s(p^e - 1) + 1}$ is generated by products of the form $f_1^{a_1} \cdots f_s^{a_s}$, where $\sum a_i = s(p^e - 1) + 1$; in particular, for some $i$ we must have $a_i \geq p^e$. Since $F$ induces Frobenius on $S / \fm S$, this shows that $J^{s(p^e - 1) + 1} \sq (F^{e} (J) ) + \fm S$, and therefore $J^{(m+1)(s(p^e-1) + 1)} \sq (F^{e} (J) )$.

	It remains to show that for a fixed $k \geq 0$ there exists some $e \gg 0$ such that $( F^{e} (J) ) \sq J^k$. Replacing $S$ by~$S / J^k$, we may assume that $J^k = 0$, and we are trying to show that $( F^{e} (J) ) = 0$ for some $e \gg 0$.

	Pick $i$ large enough so that $p^i \geq k$, and thus $J^{p^i} = 0$. Since $F$ induces Frobenius in $S / \fm S$, we get $( F^{i} (J) ) \sq J^{p^i} + \fm S = \fm S$. On the other hand, $( F^{i} (J) ) \sq J$ by assumption, so $( F^{i}  (J) ) \sq J \cap \fm S$. 

	We claim that $J \cap \fm S = J \cdot \fm S$, and therefore that $( F^{i} (J) ) \sq J \cdot \fm S$. To see this, consider the short exact sequence 
	$$
		0 \lra \fm \lra V \lra V / \fm \lra 0. 
	$$
	By assumption, $S$ and $S / J$ are flat over $V$. Applying $(-) \otimes_V S$ to this short exact sequence, we see that the natural map $\fm \otimes_V S \to \fm S$ is an isomorphism. Applying $(-) \otimes_V (S / J)$, we then get
	$$
		0 \lra \fm S \big / J \cdot \fm S \lra S / J,
	$$
	and the injectivity of this arrow tells us that $J \cap \fm S = J \cdot \fm S$. 

	Using the claim, we obtain that $(F^{ij} (J)) \sq J \cdot \fm^j S$ for all $j \geq 1$, and thus $( F^{i(m+1)} (J) ) = 0$.
\end{proof}

\begin{proposition} \label{prop-diffops-Frobenius}
	Let $(V, \fm, \KK, F)$ be as in Setup~{\rm\ref{setup-V}}. Let $R$ be a $V$-algebra of finite type and $F\colon R \to R$ be a compatible lift of Frobenius. Then we have
	$$\cD_{R|V} = \bigcup_{e = 0}^\infty \cD^{(F,e)}_R.$$
\end{proposition}
\begin{proof}
	In view of Lemma~\ref{lemma-diffops-level-ann}, it suffices to show that the families $\{J_{R|V}^n \}$ and $\{( (F \otimes F)^{e} (J_{R|V}) ) \}$ are cofinal in $R \otimes_V R$. Since $R$ is of finite type by assumption, there exist a polynomial ring $P \ceq V[x_1, \dots, x_n]$ over $V$ and a surjective $V$-algebra homomorphism $\pi\colon P \twoheadrightarrow R$, and one can construct a compatible lift of Frobenius $F\colon P \to P$ such that $F \circ \pi = \pi \circ F$. 

	The ideal $J_{R|V} \sq R \otimes_V R$ is the image of $J_{P|V} \sq P \otimes_V P$ under the induced homomorphism $\pi \otimes \pi\colon P \otimes_V P \twoheadrightarrow R \otimes_V R$. It follows that $J^n_{R|V}$ is the image of $J^n_{P|V}$ for every $n \geq 0$, and that $( (F \otimes F)^{e}(J_{R|V}) )$ is the image of $( (F \otimes F)^{e} (J_{P|V}) )$ for every $e \geq 0$. We conclude that it suffices to prove the cofinality of $\{J^n_{P|V} \}$ and $\{ ( (F \otimes F)^{e}(J_{P|V}) ) \}$, which we do by verifying that the algebra $S \ceq P \otimes_V P$ and the ideal $J \ceq J_{P|V}$ satisfy the hypotheses of Lemma~\ref{lemma-cofinal}.

	By a standard abuse of notation, let $x_i \in S$ denote the element $1 \otimes x_i$ for every $1, \dots , n$, and let $dx_i \in S$ denote the element $dx_i \ceq 1 \otimes x_i - x_i \otimes 1$. The $V$-algebra $S$ is then a polynomial algebra in the variables $x_1, \dots , x_n, d x_1, \dots , d x_n$, and the ideal $J$ is generated by $J = (dx_1, \dots , dx_n)$. In particular, $S$ is free over $V$ and therefore flat, $J$ is finitely generated, and for every integer $k \geq 0$ we have
	$$S / J^k S = \bigoplus_{\alpha} V[x_1, \dots , x_n] (dx_1)^{\alpha_1} \cdots (dx_n)^{\alpha_n},$$
	where the sum ranges through all tuples $\alpha = (\alpha_1, \dots , \alpha_n)$ of nonnegative integers for which $\alpha_1 + \cdots + \alpha_n \leq k-1$. In particular, $S / J^k S$ is also free over $V$ for every $k \geq 1$, hence flat.
\end{proof}

This alternative description of the ring of differential operators has some important implications for us in the case where $R$ is smooth over $V$. We begin this discussion by recalling a special case of the flatness criterion by fibres.
\begin{lemma} \label{lemma-flatness-by-fibres}
	Let $(V, \fm, k)$ be as in Setup~{\rm\ref{setup-V}}, let $R$ be a $V$-algebra and $M$ be an $R$-module. Suppose that $M / \fm M$ is flat over $R / \fm R$ and that $M$ is flat over $V$. Then $M$ is flat over $R$.
\end{lemma}
\begin{proof}
See \cite[Tag~06A5]{Stacks}.
\end{proof}

\begin{lemma} \label{lemma-FeR-fgp}
	Let $(V, \fm, k)$ be as in Setup~{\rm\ref{setup-V}}. Let $R$ be a smooth $V$-algebra, and let $F\colon R \to R$ be a lift of Frobenius. For every integer $e \geq 0$, the $R$-module $F^e_* R$ is finitely generated projective.
\end{lemma}
\begin{proof}
	We begin by recalling that on $V$ we have $F^e(\fm) = \fm$ (see Remark~\ref{rmk-V}), and therefore we have a natural isomorphism $F^e_* R / \fm F^e_* R \cong F^e_* ( R / \fm R)$. Since $R$ is of finite type over $V$, $R / \fm R$ is of finite type over $k$ and therefore $F$-finite; \textit{i.e.}~$F^e_* (R / \fm R) = F^e_* R / \fm F^e_* R$ is finitely generated over $R / \fm R$ for every integer $e \geq 0$. In particular, for a fixed $e \geq 0$, there exists some $N > 0$ together with a surjection 
	$$
		(R / \fm R)^{\oplus N} \overset{\pi}\lra  F^e_* R / \fm F^e_* R.
	$$
	By lifting this morphism to $R$ and letting $Q$ be the cokernel, we obtain an exact sequence
	$$
		R^{\oplus  N} \overset{\tilde \pi}\lra  F^e_* R \lra Q \lra 0.
	$$
	But, since $\tilde \pi$ is surjective modulo $\fm$, we get that $Q / \fm Q = 0$, and therefore
	$$Q = \fm Q = \fm^2 Q = \cdots = \fm^{m+1} Q = 0,$$
	where as usual $m$ is chosen so that $\fm^{m+1} = 0$. We conclude that $\tilde \pi$ is surjective and therefore that $F^e_* R$ is finitely generated.

	The smoothness of $R$ over $V$ also implies that $R / \fm R$ is smooth over $\KK$, and miracle flatness states that $F^e_* R / \fm F^e_* R$ is also projective over $R / \fm R$; see \cite[Tag~00R4]{Stacks} (this can also be seen from Kunz's theorem). Since $F$ is bijective on $V$ by assumption, the flatness of $R$ over $V$ implies that $F^e_* R$ is also flat over $V$ for every $e \geq 0$. From Lemma~\ref{lemma-flatness-by-fibres} we conclude that $F^e_* R$ is flat over $R$. Since we already know it is finitely generated, we conclude that it is projective.
\end{proof}

\begin{lemma} \label{lemma-compo-isom}
	Let $R$ be a noetherian ring, let $G$ and $H$ be finitely generated projective $R$-modules, and let $A \ceq \End_R(G)$ and $B \ceq \End_R(H)$. The morphism
	$$
	\Hom_R(G, H) \otimes_A \Hom_R(H, G) \longrightarrow
	\Hom_R(H, H)
	$$
	given by $[\phi \otimes \psi \mapsto \phi \circ \psi]$ is an isomorphism of\, $(B,B)$-bimodules.
\end{lemma}
\begin{proof}
	It suffices to check the claim locally at primes of $R$, and hence we may assume that $R$ is local; in particular, $G$ and $H$ are free. With these assumptions, the identity map on $H$ can be written as $\id_H = \sum_t \alpha_t \circ \beta_t$, where $\alpha_t\colon G \to H$ and $\beta_t\colon H \to G$ are $R$-linear homomorphisms. One now easily checks that $[\sum_t \alpha_t \otimes (\beta_t \circ \xi) \leftarrow \xi]$ is a two-sided inverse for the given map.
\end{proof}

\begin{proposition}[Frobenius descent] \label{prop-Frob-desc}
	Let $(V, \fm, \KK, F)$ be as in Setup~{\rm\ref{setup-V}}, and let $R$ be a smooth $V$-algebra equipped with a lift of Frobenius $F\colon R \to R$. For every integer $e \geq 0$, the functor $F^{e*}$ defines an equivalence of categories
	\[
		R \mmod \xrightarrow[\highsim]{F^{e*}}  \cD^{(F, e)}_R \mmod
	\]
	with inverse given by $\cC^{(F,e)}_R \ceq \Hom_R(F^e_* R, R) \otimes_{\cD^{(F, e)}_R} (-)$.
\end{proposition}
\begin{proof}
	Recall that, given an $R$-module $M$, we have
	\[
		F^{e*}(M) = \Hom_R(R, F^e_* R) \otimes_R M.
	\]
	Since $\Hom_R(R, F^e_* R)$ has a natural $(\cD^{(F, e)}_R, R)$-bimodule structure, $F^{e*}(M)$ acquires a natural left $\cD^{(F, e)}_R$-module structure. Similarly, $\Hom_R(F^e_* R, R)$ is a $(R, \cD^{(F, e)}_R)$-bimodule, and therefore $\cC^{(F, e)}_R$ as defined above gives a well-defined functor $\cD^{(F, e)}_R \mmod \to R\mmod$. 

	By Lemma~\ref{lemma-compo-isom} we have isomorphisms
	\[
		\Hom_R(F^e_* R, R) \otimes_{\cD^{(F, e)}_R} \Hom_R( R, F^e_* R ) \cong \Hom_R(R, R) \cong R
	\]
	of $(R, R)$-bimodules, which shows that $\cC^{(F, e)}_R \circ F^{e*}\colon R\mmod \to R\mmod$ is naturally equivalent to the identity functor on $R \mmod $. Similarly, we have isomorphisms
	\[
		\Hom_R(R, F^e_* R ) \otimes_{R} \Hom_R( F^e_* R , R) \cong \Hom_R( F^e_* R, F^e_* R) \cong \cD^{(F, e)}_R
	\] 
	of $(\cD^{(F, e)}_R, \cD^{(F, e)}_R)$-bimodules, which gives that $F^{e*} \circ \cC^{(F, e)}_R$ is naturally equivalent to the identity functor on $\cD^{(F, e)}\mmod$. 
\end{proof}

Let $(V, \fm, \KK, F)$ be as in Setup~\ref{setup-V}, let $R$ be a smooth algebra, let $F\colon R \to R$ be a compatible lift of Frobenius, and let $f \in R$ be a nonzerodivisor. For every integer $e \geq 0$, there is a natural $R$-module isomorphism $F^{e*}(R[f^{-1}]) \cong R[f^{-1}]$ (see the proof of Proposition~\ref{prop-Frob-lift-localization}), and, given an $R$-submodule $I \sq R[f^{-1}]$, we will identify $F^{e*}(I)$ with its image under the composition
$$F^{e*}(I) \lra F^{e*}\left(R\left[f^{-1}\right]\right) \overset{\lowsim}\lra R\left[f^{-1}\right];$$
in particular, $F^{e*}(I)$ is a new $R$-submodule of $R[f^{-1}]$. When $I \sq R$ is an ideal, this means we identify $F^{e*}(I)$ with the ideal $(F^e(I))$ given in Definition~\ref{def-expansion-ideal-Frobenius} (see Remark~\ref{rmk-ideal-F-pullback}).  

Note that every $R$-module homomorphism $\phi\colon F^{e*}R \to R$ has a unique extension to an $R$-module homomorphism $F^{e*} (R[f^{-1}]) \to R[f^{-1}]$, which we denote by $\phi'$. The assignment $[\phi \mapsto \phi']$ gives rise to an injective homomorphism
$$\Hom_R\left(F^e_* R, R\right) \lra \Hom_R\left( F^e_* \left(R\left[f^{-1}\right]\right), R\left[f^{-1}\right]\right).$$

\begin{definition} \label{def-cartier-ideal}
	Let $(V, \fm, \KK, F)$ be as in Setup~\ref{setup-V}, let $R$ be a smooth $V$-algebra, let $F\colon R \to R$ be a lift of Frobenius, and let $f \in R$ be a nonzerodivisor. Given an $R$-submodule $J \sq R[f^{-1}]$ and an integer $e \geq 0$, we define the $R$-submodule $\cC^{(F,e)}_R (J) \sq R[f^{-1}]$ to be
	\begin{align*}
		\cC^{(F,e)}_R (J) & \ceq \im \left( \Hom_R\left(F^e_* R, R\right) \otimes_R J \lra R\left[f^{-1}\right] \right) \\
			& = \left( \phi'(g) \ | \ \phi \in \Hom_R\left(F^e_* R, R\right), \ g \in J \right).
	\end{align*}
\end{definition}

\begin{corollary} \label{cor-frob-ideal-corresp}
	For every integer $e \geq 0$, the assignment $[I \mapsto F^{e*}(I)]$ defines a bijective inclusion preserving correspondence
	\[
		\left\{ \text{$R$-submodules of\, $R[f^{-1}]$}\right\}
        \overset{\lowsim}\lra        
	        \left\{\text{ $\cD^{(F,e)}_R$-submodules of\, $R[f^{-1}]$ }\right\}.
	\]
	Given a $\cD^{(F, e)}_R$-submodule $J \sq R[f^{-1}]$, the corresponding $R$-submodule is given by $\cC^{(F,e)}_R (J)$.
\end{corollary}
\begin{proof}
	Under the correspondence of Proposition~\ref{prop-Frob-desc}, the $R$-module $R[f^{-1}]$ gets exchanged with the $\cD^{(F, e)}_R$-submodule given by $F^{e*}(R[f^{-1}]) = R[f^{-1}]$. The first statement then follows from Proposition~\ref{prop-Frob-desc}. For the second statement, simply observe that $\cC^{(F,e)}_R (J)$ gets identified with the image of the composition
	\[
		\Hom_R(F^e_* R, R) \otimes_{\cD^{(F,e)}_R} J \lra \Hom_R(F^e_* R, R) \otimes_{\cD^{(F,e)}_R} R\left[f^{-1}\right] \cong R\left[f^{-1}\right],
	\]
	which is readily checked to agree with the morphism $[\phi \otimes g \mapsto \phi'(g)]$.
\end{proof}

\begin{remark} \label{rmk-1}
	Although in Corollary~\ref{cor-frob-ideal-corresp} the submodule $J$ is assumed to be a $\cD^{(F,e)}_R$-submodule, the definition of $\cC^{(F,e)}_R(J)$ given in Definition~\ref{def-cartier-ideal} applies to any $R$-submodule $J \sq R[f^{-1}]$. This is not really extra generality: given an $R$-submodule $J \sq R[f^{-1}]$, we have
	\[
		\cC^{(F,e)}_R (J) = \cC^{(F,e)}_R \left( \cD^{(F,e)}_R \cdot J \right),
	\]
	where $\cD^{(F,e)}_R \cdot J$ is the $\cD^{(F,e)}_R$-submodule of $R[f^{-1}]$ generated by $J$.
\end{remark}

\begin{corollary} \label{cor-frob-cartier-inverse}
	Let $I \sq R[f^{-1}]$ be an $R$-submodule, and let $e, i \geq 0$ be integers. We then have
	$$\cC^{(F, e)}_R (I) = \cC^{(F, e+i)}_R \left( F^{i*}(I) \right).$$
\end{corollary}
\begin{proof}
	By Remark~\ref{rmk-1}, applying the functor $F^{(e+i)*}$ to either side yields the same $\cD^{(F, e+i)}_R$-submodule of $R[f^{-1}]$, namely $\cD^{(e+i)}_R \cdot F^{i*}(I)$. The statement then follows from Corollary~\ref{cor-frob-ideal-corresp}.
\end{proof}

\begin{remark}
	When $V = \KK$ is a perfect field of positive characteristic, the submodules $\cC^{(F, e)}_R (J)$ appear in the theory of test ideals, often with the notation $J^{[1/p^e]}$. Indeed, if $R$ is a smooth $\KK$-algebra and $f \in R$ is an element, then for all integers $n \geq 0$ and $e \geq 0$, the ideal $\cC^{(F,e)}_R (f^n)$ is the test ideal of $f$ with exponent~$n / p^e$,~\textit{i.e.} 
	\[
		\tau\left(f^{n / p^e}\right) = \cC^{(F,e)}_R (f^n).
	\]
	More generally, for an arbitrary ideal $J \sq R$ and an arbitrary real number $\lambda \geq 0$, the test ideal of $J$ with exponent $\lambda$ can be expressed as
	\[
		\tau\left(J^\lambda\right) = \bigcup_{e = 0}^\infty \cC^{(F,e)}_R \left(J^{\lceil \lambda / p^e \rceil }\right);
	\]
	see \cite{BMSm2008} for details. 
\end{remark}

Note that Corollary~\ref{cor-frob-ideal-corresp} is still meaningful when $f = 1$; in this case, it tells us that $[I \mapsto F^{e*}(I)]$ and $[J \leftarrow \cC^{(F, e)}_R (J)]$ give a bijective correspondence
\[
\left\{\text{ Ideals of $R$ }\right\}
\overset{\lowsim}\longleftrightarrow
		\left\{\text{ $\cD^{(F,e)}_R$-submodules of $R$ }\right\}.
\]

The following proposition tells us how, in this situation, we can compute the ideals $\cC^{(F,e)}_R (J)$ from a basis for $F^e_* R$ and a set of generators of $J$. This generalizes the case of positive characteristic as in \cite[Proposition~2.5]{BMSm2008} (where the ideal $\cC^{(F,e)}_R (J)$ is called $J^{[1/p^e]}$).

\begin{proposition} \label{prop-cartier-comput}
	Let $(V, \fm, \KK, F)$ be as in Setup~{\rm\ref{setup-V}}, $R$ be a finite-type $V$-algebra, $F\colon R \to R$ be a compatible lift of Frobenius, and $J \sq R$ be an ideal of\, $R$. Let $e \geq 0$ be an integer, assume that $F^e_* R$ is free over $R$ with basis $F^e_* v_1, \dots , F^e_* v_N$, and fix generators $J = (f_1, \dots , f_s)$ for the ideal $J$. For each $i = 1, \dots , s$ and $j = 1, \dots , N$, let $g_{ij} \in R$ be the coefficients of\, $F^e_* f_i$ in the basis $v_j$; \textit{i.e.} 
	\begin{equation} \label{eqn-Fe-basis}
		F^e_* f_i = \sum_{j = 1}^N g_{ij} F^e_* v_j.
	\end{equation}
	Then we have
	\[
		\cC^{(F,e)}_R (J) = \left( g_{ij} \ \middle\vert \ i = 1, \dots , n; \ j = 1, \dots , N \right).
	\]
\end{proposition}
\begin{proof}
	Given an arbitrary element of $J$, say $h = \sum_j a_j f_j$, and an arbitrary $R$-linear map $\phi\colon F^e_* R \to R$, we get
	\begin{align*}
		\phi(F^e_* h) & = \sum_{j} \phi\left(F^e_* a_j f_j\right) \\
			      & = \sum_{i,j} \phi \left(g_{ij} F^e_* a_j f_j\right) \\
			      & = \sum_{i,j} g_{ij} \phi\left(F^e_* a_j f_j\right),
	\end{align*}
	which proves the inclusion $\sq$. For the other inclusion, it suffices to show that every $g_{ij}$ belongs to the left-hand side. For this, let $\pi_j: F^e_* R \to R$ be the projection to the basis element $F^e_* v_j$. One has 
	\[
		g_{ij} = \pi_j (f_i) \in \cC^{(F,e)}_R(J). \qedhere
	\]
\end{proof}

Suppose $R \ceq V[x_1, \dots, x_n]$ is a polynomial ring over $V$, and fix the compatible lift of Frobenius $F\colon R \to R$ with $F(x_i) = x_i^p$. For every integer $e \geq 0$, the $R$-module $F^e_* R$ is free, and the monomials $F^e_* x_1^{a_1} \cdots x_n^{a_n}$ with $0 \leq a_i < p^e$ form a basis. 

We fix the standard grading on $V[x_1, \dots , x_n]$. Given a (not necessarily homogeneous) polynomial $g \in R$ and an integer $d \geq 0$, we will say that $g$ has degree at most
$d$ whenever all its homogeneous components have degrees at most $d$; we write this as $\deg g \leq d$. More generally, we say that an ideal $J \sq R$ is generated in degrees at most $d$ if there are generators $J = (g_1, \dots , g_s)$ such that each $g_i$ has degree at most $d$. Note that this is equivalent to the condition
\[
	J = R \cdot \left( J \cap [R]_{\leq d} \right).
\]
In particular, if $J_1, J_2 \sq R$ are ideals generated in degrees at most $d$ and $J_1 \cap [R]_{\leq d} = J_2 \cap [R]_{\leq d}$, then $J_1 = J_2$.

We now give a useful bound on the degrees of generators of the ideals $\cC^{(F,e)}_R (J)$ which will be of use later (once again, this appears in \cite{BMSm2008} in the case of positive characteristic).

\begin{proposition} \label{prop-cartier-degrees}
	Let $(V, \fm, \KK, F)$ be as in Setup~{\rm\ref{setup-V}}, let $R \ceq V[x_1, \dots , x_n]$ be a polynomial ring over $V$, and fix $F\colon R \to R$ to be the compatible lift of Frobenius with $F(x_i) = x_i^p$. If\, $J \sq R$ is an ideal generated in degrees at most $d$, then the ideal $\cC^{(F,e)}_R (J)$ is generated in degrees at most $d / p^e$.
\end{proposition}
\begin{proof}
For the given lift of Frobenius, an $R$-basis for $F^e_* R$ is given by the monomials
$$x^\alpha \ceq x_1^{\alpha_1} \cdots x_n^{\alpha_n},$$
where $\alpha = (\alpha_1 , \dots , \alpha_n)$ is a tuple of integers with $0 \leq \alpha_i < p^e$ for every $i$. Pick generators $J = (f_1, \dots , f_s)$ such that $\deg f_i \leq d$, and express each $F^e_* f_i$ in this basis as in Equation (\ref{eqn-Fe-basis}). This gives
$$f_i = \sum_{0 \leq \alpha_i < p^e} F^e(g_{i, \alpha}) x^\alpha.$$
Since the given lift of Frobenius multiplies degrees by $p^e$, for every $i$ and $\alpha$, we get
$$
d \ \geq \ \deg(f_i) \ \geq \ p^e \deg(g_{i \alpha}) + \deg (x^\alpha) \ \geq \ p^e \deg (g_{i \alpha}),
$$
from which we deduce that $\deg(g_{i \alpha}) \leq d / p^e$. Since $\cC^{(F, e)}_R (J)$ is generated by the $g_{i \alpha}$, the statement follows. 
\end{proof}

\subsection{Certain algebras of continuous functions} \label{subscn-alg-cts-function}

Throughout this subsection we fix a prime number $p$, and $\Z_p$ will denote the ring of $p$-adic integers. Given sets $X$ and $Y$, we denote by $\Fun(X,Y)$ the set of all functions $X \to Y$. 

\begin{definition}
Let $A$ be a set. We denote by $C(\Z_p, A)$ the set of continuous functions $\phi\colon \Z_p \to A$, where $A$ is equipped with the discrete topology.
\end{definition}

Recall that 
$$\Z_p = \lim_{\leftarrow e} ( \Z / p^e ),$$
from which we deduce
$$C\left(\Z_p, A\right) = \lim_{\rightarrow e} \Fun( \Z / p^e, A).$$
In other words, a function $\phi\colon \Z_p \to A$ belongs to $C(\Z_p, A)$ if and only if there exists some integer $e \geq 0$ such that $\phi(\alpha) = \phi(\beta)$ for all $\alpha, \beta \in \Z_p$ with $\alpha \equiv \beta \mod p^e \Z_p$.

If $A$ is a ring (not necessarily commutative), then $C(\Z_p, A)$ inherits a ring structure, given by pointwise addition and multiplication. In this situation, given an integer $e \geq 0$ and a $p$-adic integer $\alpha \in \Z_p$, we denote by $\chi^e_\alpha \in C(\Z_p, A)$ the operator
$$\chi^e_\alpha(\beta) \ceq \begin{cases}
	1& \text{ if } \beta \equiv \alpha \mod p^e \Z_p, \\
	0& \text{ otherwise.}
\end{cases}$$
Note that these operators satisfy the identity
\begin{equation} \label{eqn-chi-trans}
	\chi^e_\alpha = \sum_{i = 0}^{p^e - 1} \chi^{e+1}_{\alpha + i p^e}.
\end{equation}
Observe that if $M$ is an $A$-module, then $C(\Z_p, M)$ is naturally a $C(\Z_p, A)$-module.

If $A$ is a commutative ring, then so is $C(\Z_p, A)$, and hence we may consider $\Spec( C(\Z_p, A))$. We give a description of this spectrum as a locally ringed space for some special $A$.

\begin{proposition} \label{prop-Spec-C}
	Suppose $V$ is as in Setup~{\rm\ref{setup-V}}. Then we have an isomorphism
	$$\Spec \left( C\left(\Z_p, V\right) \right) \cong \left( \Z_p, \underline V \right)$$
	of locally ringed spaces, where $\underline V$ is the locally constant sheaf with value $V$.
\end{proposition}

\begin{proof}
	If $e \geq 0$ is an integer, then $\Fun( \Z / p^e, V)$ is a product of $p^e$ copies of $V$, and therefore the ringed space $\Spec (\Fun(\Z / p^e, V))$ is a disjoint union of $p^e$ copies of $\Spec (V)$, indexed by $\Z / p^e$. We thus have isomorphisms of locally ringed spaces
	$$\Spec \left( \Fun(\Z / p^e, V) \right) \cong \Spec(V) \times \Z / p^e \cong \left( \Z / p^e, \underline V \right),$$
	where the last isomorphism follows because $\Spec(V)$ is a single point.
	
	By \cite[01YW]{Stacks} we thus get
	$$\Spec \left( C\left(\Z_p, V\right) \right) \cong \lim_{\leftarrow e} \left( \Z / p^e, \underline V \right),$$
	where the limit takes place in the category of locally ringed spaces. It hence suffices to show that $(\Z_p, \underline V)$ satisfies the universal property of the right-hand side.

	First observe that for each integer $e \geq 0$ we have a natural map $\pi_e\colon (\Z_p, \underline V) \to (\Z / p^e, \underline V)$. Suppose that $(W, \cO_W)$ is a locally ringed space with compatible maps $f_e\colon (W, \cO_W) \to (\Z / p^e, \underline V)$ for every $e \geq 0$. This induces a unique compatible continuous map of topological spaces $f_\infty\colon W \to \Z_p$, and the claim is that there is a unique compatible map of sheaves $\underline V \to f_{\infty *} (\cO_W)$, and hence there is a unique compatible map $f_\infty\colon (W, \cO_W) \to (\Z_p, \underline V)$.

	To prove the claim, note that a map of sheaves $\underline V \to f_{\infty, *} \cO_W$ is equivalent to the data of ring homomorphisms $V \to \cO_W( f^{-1}_{\infty} U)$ for every open $U \sq \Z_p$ (compatible with restriction), and recall that it is enough to specify these on a basis for the topology of $\Z_p$. To conclude the proof, recall that open sets of the form $\pi_e^{-1} (x)$ for $x \in \Z / p^e$ form a basis for the topology of $\Z_p$, and note that we have $f^{-1}_\infty \pi^{-1}_e(x) = f^{-1}_e (x)$.
\end{proof}

Fix $(V, \fm, \KK, F)$ as in Setup~\ref{setup-V}. By Proposition~\ref{prop-Spec-C} every prime ideal of $C(\Z_p, V)$ is maximal, and there is a bijective correspondence between $p$-adic integers and maximal ideals of $C(\Z_p, V)$. Tracing through the proof, one sees that, given a $p$-adic integer $\alpha \in \Z_p$, the corresponding maximal ideal is
$$(\fm : \alpha) \ceq \left(\phi \in C\left(\Z_p, V\right) \ \middle\vert \ \phi(\alpha) \in \fm \right).$$

This is an instance of a more general construction: given an ideal $\fa \sq V$ and a $p$-adic integer $\alpha \in \Z_p$, we can define an ideal $(\fa : \alpha) \sq C(\Z_p, V)$ given by
$$(\fa : \alpha) \ceq \left(\phi \in C\left(\Z_p, V\right) \ \middle\vert \ \phi(\alpha) \in \fa \right).$$
For example, the ideal $(0 : \alpha) \sq C(\Z_p, V)$ is the ideal that consists of functions which vanish at $\alpha$. Note that the evaluation morphism $[\phi \mapsto \phi(\alpha)]$ defines a surjective ring homomorphism $e_\alpha\colon C(\Z_p, V) \to V$ whose kernel is precisely $(0:\alpha)$, and therefore $V \cong C(\Z_p, V) / (0 : \alpha)$. Given an ideal $\fa \sq V$, we have $e_\alpha^{-1}(\fa) = (\fa : \alpha)$. 

This construction is compatible with taking powers, in the sense that $(\fa : \alpha)^k = (\fa^k : \alpha)$ for every ideal $\fa \sq V$, every $p$-adic integer $\alpha \in \Z_p$, and every integer $k \geq 0$. In particular, if $m \geq 0$ is chosen so that $\fm^{m+1} = 0$ in $V$, then we have
$$(\fm : \alpha) \supseteq (\fm : \alpha)^2  \supseteq \cdots \supseteq (\fm : \alpha)^{m} \supseteq (\fm : \alpha)^{m+1} = (\fm : \alpha)^{m+2} = \cdots = (0 : \alpha).$$

Given a $C(\Z_p, V)$-module $M$, we obtain an associated quasicoherent sheaf on $\Spec( C(\Z_p, V)) \cong \Z_p$. Recall that, given $\alpha \in \Z_p$, the stalk $M_\alpha$ of this sheaf at $\alpha \in \Z_p$ is the localization of $M$ at the maximal ideal $(\fm: \alpha)$. 

\begin{proposition} \label{prop-C-stalk-description}
	Fix $(V, \fm, \KK, F)$ as in Setup~{\rm\ref{setup-V}}, let $M$ be a $C(\Z_p, V)$-module, and let $\alpha \in \Z_p$ be a $p$-adic integer. There is a natural $C(\Z_p, V)$-linear isomorphism
	$$M_\alpha \cong M / (0:\alpha) M.$$
\end{proposition}
\begin{proof}
	For ease of notation, let $C \ceq C(\Z_p, V)$. Since $M_\alpha \cong M \otimes_C C_\alpha$ and $M / (0:\alpha) M \cong M \otimes_C C / (0:\alpha)$, it suffices to prove the claim for $M = C$. Given an element $\phi \in C$, we denote by $[\phi]$ its equivalence class in $C / (0:\alpha)$. 

	Suppose $\psi \in C$ is such that $\psi \notin (\fm:\alpha)$; in other words, $\psi(\alpha) \notin \fm$. It follows that $\psi(\alpha)$ is invertible in~$V$. We have $\psi(\alpha)^{-1} \psi - 1 \in (0:\alpha)$, and therefore $[\psi(\alpha)^{-1}][\psi(\alpha)] = [\psi(\alpha)^{-1} \psi] = 1$. This shows that $\psi$ is invertible in $C / (0:\alpha)$. Since $\psi \in C \setminus (\fm:\alpha)$ was arbitrary, we get a natural $C$-linear map $C_\alpha \to C / (0:\alpha)$.

	Now let $\phi \in C$ be an element such that $\phi \in (0:\alpha)$, \textit{i.e.}~$\phi(\alpha) = 0$. Picking an $e \gg 0$ such that $\phi \in \Fun(\Z / p^e, V)$, we see that $\chi^e_\alpha \phi = 0$. Since $\chi^e_\alpha \notin (\fm:\alpha)$, we have $\phi / 1 = 0$ in $C_\alpha$. This shows that the natural map $C \to C_\alpha$ factors through $C / (0:\alpha)$, giving a natural map $C / (0:\alpha) \to C_\alpha$. 

	Since the maps $C_\alpha \to C / (0:\alpha)$ and $C / (0:\alpha) \to C_\alpha$ obtained above commute with the natural maps $C \to C_\alpha$ and $C \to C / (0:\alpha)$, we conclude that they are mutual inverses.
\end{proof}

\begin{corollary} \label{cor-C-double-stalk}
	If $\beta \in \Z_p$ is another $p$-adic integer, we have
	$$(M_\alpha)_\beta \cong 
	\begin{cases}
		M_\alpha& \text{ if } \alpha = \beta, \\
		0& \text{ otherwise. }
	\end{cases}
	$$
\end{corollary}
\begin{proof}
	Write $M_\alpha \cong M / (0:\alpha) M$. Since $(\fm:\alpha)$ is the only maximal ideal containing $(\fm:\alpha)^{m+1} = (0:\alpha)$, we conclude that $(M_\alpha)_\beta = 0$ whenever $\alpha \neq \beta$. On the other hand, it is clear that $(M_\alpha)_\alpha \cong M_\alpha$. 
\end{proof}

\begin{proposition} \label{prop-disc-C-module}
	Fix $(V, \fm, \KK, F)$ as in Setup~{\rm\ref{setup-V}}, and let $M$ be a $C(\Z_p, V)$-module. Suppose there are only finitely many $\alpha \in \Z_p$ such that $M_\alpha \neq 0$; call these $\alpha_1, \dots , \alpha_s$. Then the natural map
	$$M \lra M_{\alpha_1} \oplus M_{\alpha_2} \oplus \cdots \oplus M_{\alpha_s}$$
	is an isomorphism, and the annihilator of\, $M$ takes the form
	$$\Ann(M) = (\fa_1 : \alpha_1) \cdot (\fa_2 : \alpha_2) \cdots (\fa_s : \alpha_s)$$
	for some ideals $\fa_1, \dots , \fa_s \sq V$.
\end{proposition}
\begin{proof}
	Let $K$ (resp.\ $Q$) be the kernel (resp.\ cokernel) of the given map, so that we have an exact sequence
	$$0 \lra K \lra M \lra M_{\alpha_1} \oplus \cdots \oplus M_{\alpha_s} \lra Q \lra 0;$$
	the goal is to show that $K = Q = 0$. For this, take an arbitrary $\beta \in \Z_p$, and take the stalk at $\beta$ of this sequence. By virtue of Corollary~\ref{cor-C-double-stalk}, we conclude that, whenever $\beta \in \{\alpha_1, \dots , \alpha_s \}$, we get
	$$0 \lra K_\beta \lra M_\beta \lra M_\beta \lra Q_\beta \lra 0.$$
	On the other hand, when $\beta \notin \{\alpha_1, \dots , \alpha_s \}$, we get
	$$0 \lra K_\beta \lra 0 \lra 0 \lra Q_\beta \lra 0.$$
	In either case we get $K_\beta = Q_\beta = 0$, and since this happens for all $\beta \in \Z_p$, we conclude that $K = Q = 0$ as claimed.  

	For the statement about the annihilator, observe that for every $i = 1, \dots , s$ the action of $C(\Z_p, V)$ on $M_{\alpha_i}$ factors through the evaluation map $e_\alpha\colon C(\Z_p, V) \to C(\Z_p, V) / (0: \alpha_i) \cong V$, and therefore the annihilator of $M_{\alpha_i}$ has the form $e_{\alpha_i}^{-1}(\fa_i) = (\fa_i : \alpha_i)$ for some ideal $\fa_i \sq V$. We conclude that the annihilator of $M$ is $\bigcap_i (\fa_i : \alpha_i)$, and the statement follows from the fact that the ideals $(\fa_i : \alpha_i)$ are pairwise coprime.
\end{proof}

For us the algebra $C(\Z_p, V)$ comes into play through rings of differential operators as follows. We let $(V, \fm, \KK, F)$ be as in Setup~\ref{setup-V}, and consider the polynomial ring $V[t]$ in the variable $t$. We give $V[t]$ the natural $\N$-grading for which $\deg(t) = 1$; this induces a $\Z$-grading on the ring $\cD_{V[t]}$ of $V$-linear differential operators on $V[t]$. We denote by $(\cD_{V[t]})_0$ the subring of $\cD_{V[t]}$ that consists of degree zero differential operators.

If $\phi \in C(\Z_p, V)$ is a continuous function $\phi\colon \Z_p \to V$, we let $\tilde \phi\colon V[t] \to V[t]$ be the unique $V$-linear map given by
$$\tilde \phi (t^a) = \phi(-1-a) t^a$$
for every integer $a \geq 0$. 

We claim that $\tilde \phi$ is a differential operator. To see this, we fix the compatible lift of Frobenius $F\colon V[t] \to V[t]$ given by $F(t) = t^p$. If $e \gg 0$ is large enough so that $\phi$ factors through $\Z / p^e$, then for every integer $a \geq 0$ we get
\begin{align*}
	\tilde \phi (F^e(t) t^a) & = \tilde \phi \left(t^{p^e + a}\right) = \phi(-1-a-p^e) t^{p^e + a} \\
		& = \phi(-1-a) t^{p^e + a} = t^{p^e} \tilde \phi (t^a) = F^e(t) \tilde \phi (t^a),
\end{align*}
which shows that $\tilde \phi \in \cD^{(F, e)}_{V[t]}$. Since $\tilde \phi$ clearly has degree zero, we conclude that $\tilde \phi \in (\cD^{(F, e)}_{V[t]})_0$. 

We therefore get a $V$-algebra homomorphism 
$$ 
	C\left(\Z_p, V\right) \overset{\Delta}{\underset{\highsim}\lra}  \left(\cD_{V[t]}\right)_0,
$$
given by $\Delta(\phi) = \tilde \phi$, which one immediately checks to be an isomorphism. 

Note, for example, that given an integer $e \geq 0$ and integers $0 \leq a, b < p^e$, the operator associated to the function $\phi = \chi^e_a$ given above acts by
$$\tilde \chi^e_a \left(t^{p^e - 1 -b}\right) = \begin{cases}
	t^{p^e - 1 - a} &\text{ if } b = a, \\
	0 &\text{ otherwise.}
\end{cases}$$

We also have the following generalization. Suppose $(V, \fm, \KK, F)$ are as in Setup~\ref{setup-V} and $R$ is a finite-type $V$-algebra. We consider the polynomial ring $R[t]$ in the variable $t$, to which we give the natural $\N$-grading by $\deg(R) = 0$ and $\deg(t) = 1$. This induces a $\Z$-grading on the ring $\cD_{R[t]}$ of $V$-linear differential operators on $R[t]$. The isomorphism $C(\Z_p, V) \cong (\cD_{V[t]})_0$ can be upgraded to an isomorphism
$$
	C\left(\Z_p, \cD_R\right) \overset{\Delta}{\underset{\highsim}\lra}  \left(\cD_{R[t]}\right)_0. 
$$
Indeed, this isomorphism is obtained as the composition
\begin{align*}
	C\left(\Z_p, \cD_R\right) & \cong C\left(\Z_p, V\right) \otimes_V \cD_R \overset{\lowsim}\lra \left(\cD_{V[t]}\right)_0 \otimes_V \cD_R \cong \left(\cD_{R[t]}\right)_0.
\end{align*}

\section{Definition and characterization of Bernstein--Sato roots} \label{scn-b-function-Zpm}

\subsection{Motivation from characteristic zero} \label{subscn-char0-mot}

Let $R \ceq \C[x_1, \dots , x_n]$ be a polynomial ring over $\C$. Recall that, in this case, the ring $\cD_{R}$ of $\C$-linear differential operators on $R$ can be described as the following subalgebra of $\End_\C(R)$:
$$\cD_R = R \la \partial_1, \dots , \partial_n \ra,$$
where $\partial_i$ is the operator $\partial_i \ceq \partial / \partial x_i$. 

Let $f \in R$ be a nonzero polynomial, let $s$ be a new indeterminate, and consider the $\cD_R[s]$-module $R[f^{-1} , s] \bs{f^s}$, where $\bs{f^s}$ is a formal symbol. The action of $R$ and $s$ on $R[f^{-1}, s] \bs{f^s}$ is as expected; in other words, we have $R[f^{-1}, s] \bs{f^s} \cong R[f^{-1}, s]$ as an $R[s]$-module. However, the action of derivations is determined~by
$$\partial_i \cdot \bs{f^s} \ceq s \frac{\partial f}{\partial x_i} f^{-1} \bs{f^s}.$$
A theorem of Bernstein and Sato states that there is a nonzero polynomial $b(s) \in \C[s]$ such that
$$b(s) \bs{f^s} = P(s) f \bs{f^s}$$
for some $P(s) \in \cD_R[s]$. The monic polynomial of least degree satisfying the above for some $P(s) \in \cD_R[s]$ is called the Bernstein--Sato polynomial of $f$ and is denoted by $b_f(s)$. Said differently, $b_f(s)$ is the monic generator of the ideal
$$\left(b_f(s)\right) = \Ann_{\C[s]} \left( \cD_R[s] \cdot \bs{f^s} \ \big/ \ \cD_R[s] \cdot f \bs{f^s} \right) \sq \C[s].$$

As explained in Section~\ref{scn-intro}, this description is not suitable for translation to positive characteristic, and one uses instead an alternative construction of $b_f(s)$ on which we now elaborate. We note that the earliest reference to this construction we could find is in work of Malgrange \cite{Mal83}.

Consider the graph $\gamma_f\colon \A^n \to \A^n \times \A^1$ of the polynomial $f$, and the $\cD_R$-module pushforward $\gamma_{f,+} (R[f^{-1}])$ of the $\cD_R$-module $R[f^{-1}]$. This admits a concrete description as
%
$$\gamma_{f,+} \left(R\left[f^{-1}\right]\right) = \frac{R\left[t, f^{-1}, (f-t)^{-1}\right] }{  R\left[t, f^{-1}\right]}, $$
where $t$ is the coordinate on the new copy of $\A^1$, and for each integer $k \geq 0$ we let $\delta_k$ denote the class
$\delta_k \ceq [(f-t)^{- 1 - k}]\in \gamma_{f,+} (R[f^{-1}]$.
These elements $\delta_k$ form a basis for $\gamma_{f,+}(R[f^{-1}])$ as an $R[f^{-1}]$-module; that is,
$$\gamma_{f,+} \left(R\left[f^{-1}\right]\right) = \bigoplus_{ k = 0 }^\infty R\left[f^{-1}\right] \ \delta_k.$$
On the other hand, the module $R[f^{-1}, s] \bs{f^s}$ is also free over $R[f^{-1}]$, with a basis given by
$$R\left[f^{-1}, s\right] \bs{f^s} = \bigoplus_{ i = 0 }^\infty R\left[f^{-1}\right] s^i \bs{f^s}.$$
With this notation, there is an isomorphism
$$
	\gamma_{f,+} \left(R\left[f^{-1}\right]\right) \overset{\Phi}{\underset{\highsim}\lra} R\left[f^{-1}, s\right] \bs{f^s}
$$
which is $\cD_R[f^{-1}]$-linear and exchanges $\delta_0$ with the action of $- \partial_t t$ on the left with $s$ on the right. In particular, the change of basis is given by
$$\Phi(\delta_k) = (-1)^k \binom{s}{k}\, f^{-k} \bs{f^s}.$$
In view of this, the Bernstein--Sato polynomial of $f$ can also be viewed as the annihilator
$$\left( b_f (s)\right) = \Ann_{\C[s]} \left( \cD_R[- \partial_t t] \cdot \delta_0 \ \big/ \ \cD_R[-\partial_t t] \cdot f \delta_0 \right) \sq \C[s],$$
where $s$ acts on the module as the operator $- \partial_t t$. Since the module $\gamma_{f,+} (R[f^{-1}])$ can be constructed over an arbitrary base, this latter description is better for generalization, as long as one uses the correct replacement of $\cD_R[- \partial_t t]$. We refer the reader to \cite{Mustata2009,Bitoun2018,QG19,JNBQG} for an implementation of these ideas in positive characteristic.

\subsection{The \texorpdfstring{$\boldsymbol{b}$}{b}-function and Bernstein--Sato roots over \texorpdfstring{$\boldsymbol{\Z / p^m}$}{Z/pm}} \label{subscn-malgrange-module-b-function}

Let $(V, \fm, \KK, F)$ be as in Setup~\ref{setup-V}, let $R$ be a finite-type $V$-algebra equipped with a compatible lift of Frobenius $F\colon R \to R$, and let $f \in R$ be a nonzerodivisor. Consider the polynomial ring $V[t]$ on a variable $t$, equipped with the lift of Frobenius given by $F(t) = t^p$. Note that the lifts of Frobenius on $R$ and $V[t]$ induce one on $R[t] = R \otimes_V V[t]$. With this notation, for all integers $e \geq 0$, we have
$$\cD^{(F, e)}_{R[t]} = \cD^{(F, e)}_R \otimes_V \cD^{(F, e)}_{V[t]}.$$
Consider the $\cD_{R[t]}$-module
$$H_f \ceq \frac{R\left[t, f^{-1}, (f-t)^{-1}\right]}{ R\left[t, f^{-1}\right]}$$
and the element $\delta_0 \ceq [(f - t)^{-1}] \in H_f$. Now regard $R[t]$ as a graded ring, in which $\deg(R) = 0$ and $\deg(t) = 1$. This induces a $\Z$-grading on $\cD_{R[t]}$, and we let $(\cD_{R[t]})_0$ denote its degree zero piece. 

Recall that we have $(\cD_{R[t]})_0 = \cD_R \otimes_V (\cD_{V[t]})_0$ and that we have an isomorphism $(\cD_{V[t]})_0 \cong C(\Z_p, V)$ (see Section~\ref{subscn-alg-cts-function}). In particular, every $(\cD_{R[t]})_0$-module has a natural $C(\Z_p, V)$-structure by restriction of scalars. This is the case, for example, for the module
%
%
$$N_f \ceq \frac{\left(\cD_{R[t]}\right)_0 \cdot \delta_0 }{\left(\cD_{R[t]}\right)_0 \cdot f \delta_0}.$$

Malgrange's description of the Bernstein--Sato polynomial (see Section~\ref{subscn-char0-mot}), as well as existing work in positive characteristic analogues, see \cite{Bitoun2018,QG19,JNBQG}, leads to the following as an analogue of the Bernstein--Sato polynomial in our current setting. 

\begin{definition}
	Let $(V, \fm, \KK, F)$ be as in Setup~\ref{setup-V}, let $R$ be a smooth $V$-algebra, and let $f \in R$ be a nonzerodivisor. The $b$-function of $f$ is the ideal $B_f \sq C(\Z_p, V)$ given as the annihilator
	$$B_f \ceq \Ann \left( N_f \right)  \sq C\left(\Z_p, V\right).$$
\end{definition}

\begin{remark}
	We will see (Corollary~\ref{cor-Nf-discrete}) that when $R$ is a polynomial ring over $V$, the $C(\Z_p, V)$-module $N_f$ satisfies the hypothesis of Proposition~\ref{prop-disc-C-module}; that is, there are only finitely many $\alpha \in \Z_p$ such that the stalk $(N_f)_\alpha$ is nonzero. If we call these $\alpha_1, \dots , \alpha_s$, this will imply that 
	$$N_f \cong \left(N_f\right)_{\alpha_1} \oplus \cdots \oplus \left(N_f\right)_{\alpha_s}$$
	and that
	$$B_f = (\fa_1 : \alpha_1) \cdot (\fa_2 : \alpha_2) \cdots (\fa_s : \alpha_s).$$
	To simplify the discussion, it will be useful to give these $\alpha_i$ a name at this stage.
\end{remark}

\begin{definition} \label{def-BSR}
	Let $(V, \fm, \KK, F)$ be as in Setup~\ref{setup-V}, let $R$ be a smooth $V$-algebra, and let $f \in R$ be a nonzerodivisor. A $p$-adic integer $\alpha \in \Z_p$ is called a Bernstein--Sato root of $f$ whenever $(N_f)_\alpha \neq 0$.
\end{definition}

Note that in Theorem~\ref{thm-BSR-equivalences} we obtain alternative characterizations of Bernstein--Sato roots. 

We begin our discussion by exploiting Frobenius to give an alternative description of $H_f$. For every integer $e \geq 0$, observe that $f^{p^{e}} - t^{p^e}$ is invertible in $R[t, f^{-1}, (f - t)^{-1}]$ by Proposition~\ref{prop-Frob-lift-localization}. We can thus let $\delta_{e} \in H_f$ denote the class  
$$\delta_e \ceq \left[\left(f^{p^e} - t^{p^e}\right)^{-1}\right] \in H_f.$$

Fix an integer $m \geq 0$ such that $\fm^{m+1} = 0$ in $V$. For every integer $e \geq 0$, observe that we have
$$\delta_{e + m} = \left[F^e\left(f^{p^m} - t^{p^m}\right)^{-1}\right]$$
and, moreover,
\begin{align*}
	\delta_{m + e} & = F^{e+1}\left(f^{p^m} - t^{p^m}\right) \ F^e\left(f^{p^m} - t^{p^m}\right)^{-1} \ \delta_{m+e+1} \\
		& = F^e \left( \left(f^{p^{m+1}} - t^{p^{m+1}}\right) (f^{p^m} - t^{p^m})^{-1} \right) \ \delta_{m + e+ 1} \\
		& = F^e \left( \sum_{i = 0}^{p-1} f^{i p^m} t^{(p-1-i)p^m} \right) \delta_{m+e+1} \\
		& =  \left( \sum_{i = 0}^{p-1} F^e\left(f^{ip^m}\right) \ F^{e+m} \left(t^{p-1-i}\right) \right) \delta_{m + e + 1}.
		\qedhere
\end{align*}

On the other hand, for every integer $e \geq 0$, we consider the module
$$ \tilde H^{m + e}_f \ceq \frac{R\left[t, f^{-1}\right]}{\left(f^{p^{m+e}} - t^{p^{m+ e}}\right)} \ \tilde \delta_{m + e} = \frac{R\left[t, f^{-1}\right]}{\left(F^e\left(f^{p^m}\right) - F^{e+m}(t)\right)}\tilde \delta_{m + e},$$
where $\tilde \delta_{m + e}$ is just a bookkeeping symbol. Note that the natural $\cD_{R[t]}$-module structure of $R[t, f^{-1}]$ endows $\tilde H^{m+e}_f$ with a natural $\cD^{(F, e)}_R \otimes_V \cD^{(F, m+e)}_{V[t]}$-module structure. Moreover, the homomorphism $\tilde H^{m+e}_f \to \tilde H^{m+e+1}_f$ given by
$$\tilde \delta_{m+e} \longmapsto \left( \sum_{i = 0}^{p-1} F^e\left(f^{i p^m}\right) F^{e+m} \left(t^{p-1-i}\right) \right) \ \tilde \delta_{m + e + 1}$$
%
is $\cD^{(F, e)}_R \otimes_V \cD^{(F, m+e)}_{V[t]}$-linear. It follows that the colimit of the $\tilde H^{m+e}_f$ has a natural $\cD_{R[t]}$-module structure. 

\begin{lemma} \label{lemma-cofinal-1}
	Let $(V, \fm, \KK, F)$ be as in Setup~{\rm\ref{setup-V}}, let $S$ be a $V$-algebra, and let $x, y \in S$ be two elements. The families of ideals $\{(x-y)^n \}_{n = 0}^\infty$ and $\{ x^{p^e} - y^{p^e} \}_{ e = 0}^\infty$ are cofinal.
\end{lemma}
\begin{proof}
	Fix an integer $e \geq 0$. We want to find an integer $n \geq 0$ such that $(x-y)^n \in (x^{p^e} - y^{p^e})$. Going modulo the ideal $(x^{p^e} - y^{p^e})$, we may assume that $x^{p^e} = y^{p^e}$, and we want to find an $n \geq 0$ such that $(x-y)^n = 0$. 

	Since $S / \fm S$ has characteristic $p > 0$, we know that $(x - y)^{p^e} = 0$ in $S / \fm S$; that is, we have $(x - y)^{p^e} \in \fm S$. If $m \geq 0$ is such that $\fm^{m+1} = 0$ in $V$, we conclude that $(x-y)^{(m+1) p^e} = 0$. 

	Now fix $n \geq 0$. We want to find some $e \geq 0$ such that $(x^{p^e} - y^{p^e}) \sq (x-y)^n$. First, assume that $n = p^i$ for some integer $i \geq 0$. Then, working modulo the ideal $(x-y)^{p^i}$, we may assume that $(x-y)^{p^i} = 0$, and we want to show that $x^{p^e} = y^{p^e}$ for some large enough~$e$.

	Again, since $S / \fm S$ has characteristic $p > 0$, we know that $x^{p^i} \equiv y^{p^i} \mod \fm S$. If $m \geq 0$ is such that $\fm^{m+1} = 0$ in $V$, we get that $x^{p^{m+i}} = y^{p^{m+i}}$ by Lemma~\ref{lemma-p-power}.
\end{proof}

\begin{proposition}
	With the notation as above, there is a $\cD_{R[t]}$-module isomorphism
	$$\lim_{\to e} \tilde H^{m+e}_f \overset{\lowsim}\lra H_f$$
	which, for every $e \geq 0$, identifies $\tilde \delta_{m + e}$ with $\delta_{m + e}$.  
\end{proposition}

\begin{proof}
	Note that the module $H_f$ is $(f-t)$-power torsion. By Lemma~\ref{lemma-cofinal-1} the families of ideals $\{ (f-t)^n \}$ and $\{f^{p^{m + e}} - t^{p^{m+e}} \} = \{ F^e(f^{p^m} - t^{p^m} ) \}$ are cofinal. Letting $\delta_{m+e} \ceq [F^e (f^{p^m} - t^{p^m})^{-1}] \in H_f$, we obtain
	\begin{align*}
		H_f = \bigcup_{e = 0}^\infty \Ann_{H_f} \left( F^e\left( f^{p^m} - t^{p^m} \right) \right) = \bigcup_{e = 0}^\infty R\left[t, f^{-1}\right] \ \delta_{m + e}.
	\end{align*}
	Observe that $\delta_{m + e} \in H_f$ induces a $\cD^{(F, e)}_R \otimes_V \cD^{(F, m + e)}_{V[t]}$-linear map 
	$$R\left[t, f^{-1}\right] \xrightarrow{\delta_{e+m}} H_f$$
	whose kernel is precisely $(F^e(f^{p^m} - t^{p^m}))$. To complete the proof, it suffices to observe that the transition maps agree; this follows from the discussion above.
\end{proof}

Given integers $e \geq 0$ and $0 \leq a < p^{m+e}$, let $Q^{m + e}_a \in H_f$ be the element given by
$$Q^{m+e}_a \ceq f^a \ t^{p^{m + e} - 1 - a} \ \delta_{m + e} \in H_f.$$
Note that we have
\begin{align*}
	f^a \ t^{p^{m + e} - 1 - a} \ \delta_{m + e} & = \left(f^a \ t^{p^{m + e} - 1 - a}\right) \left( \sum_{i = 0}^{p-1} f^{ip^{m + e}} \ t^{(p-1-i) p^{m + e}} \right) \delta_{m + e + 1} \\
						     & = \left( \sum_{i = 0}^{p-1} f^{a + i p^{m + e}} t^{(p-1 - a- ip^{m+e}) p^{m + e}} \right) \delta_{m + e + 1},
\end{align*}
which gives 
\begin{equation} \label{eqn-Q-trans}
	Q^{m+e}_a = \sum_{i = 0}^{p-1} Q^{m + e + 1}_{a + i p^{m + e}}.
\end{equation}
Moreover, if $P \in \cD^{(F, e)}_R$ is a differential operator of level $e$, $g \in R[f^{-1}]$ is an element, and $a,b$ are integers with $0 \leq a, b < p^{m+e}$, we have
\begin{align*}
	P \tilde \chi^{m+e}_a \cdot g Q^{m+e}_b & = P \tilde \chi^{m+e}_a \left( f^b t^{p^{m+e} - 1 - b} \ g \delta_{m + e}\right) \\
						& = \chi^{m+e}_a(b - p^{m+e})  P \left( f^b t^{p^{m+e} - 1 - b} \ g \right)\delta_{m + e},
\end{align*}
which gives
\begin{equation} \label{eqn-action-two}
	P \tilde \chi^{m+e}_a \cdot g Q^{m+e}_b =
	\begin{cases}
		f^{-a} P(f^a g) Q^{m+e}_a &\text{ if } a = b, \\
		0 &\text{ otherwise }
	\end{cases}
\end{equation}
(where the operators $\tilde \chi^{m+e}_a$ are as in Section~\ref{subscn-alg-cts-function}). 

\subsection{The module \texorpdfstring{$\boldsymbol{C(\Z_p, R[f^{-1}])\bs{f^s}}$}{C(Zp, R[f-1])fs}}

In characteristic zero it is useful to be able to understand the $b$-function through a functional equation instead of working with the module $N_f$. In this subsection we describe the module $C(\Z_p, R_f) \bs{f^s}$, which plays the role of $R_f[s] \bs{f^s}$ in our setting, and we show that it is isomorphic to the module $H_f$. From there we extract a new characterization of the $b$-function. In the case of positive characteristic, this construction was carried out in \cite{JNBQG}.

Recall that, given a set $A$, we denote by $C(\Z_p, A)$ the set of continuous functions $\Z_p \to A$, where $A$ has the discrete topology (see Section~\ref{subscn-alg-cts-function}). 

Let $(V, \fm, \KK, F)$ be as in Setup~\ref{setup-V}. Let $R$ be a finite-type $V$-algebra with a compatible lift of Frobenius $F\colon R \to R$, let $f \in R$ be a nonzerodivisor, and fix an $m > 0$ such that $\fm^{m+1} = 0$ in $V$. 

\begin{lemma} \label{lemma-conj-welldef}
	Suppose $P \in \cD^{(F, e)}_{R[f^{-1}]}$ is a differential operator of level $e$ with respect to $F$. For every $p$-adic integer $\alpha \in \Z_p$ and every element $g \in R[f^{-1}]$, the element
	$$f^{-a} P(f^a g) \in R\left[f^{-1}\right]$$
	is independent of the choice of $a \in \Z$ with $a \equiv \alpha \mod p^{e+m} \Z_p$. In particular, the function $\Z \to R[f^{-1}]$ given by $[a \mapsto f^{-a} P(f^a g)]$ is locally constant in the $p$-adic topology. 
\end{lemma}

\begin{proof}
	Let $e$ be such that $P \in \cD^{(F, e)}_R$. Suppose $a, b$ are such that $a,b \equiv \alpha \mod p^{e+m} \Z_p$. We then have $a = b + k p^{e+m}$ for some $k \in \Z$. By Lemma~\ref{lemma-Frob-p-power} we have 
	$$f^{k p^{e+m}} = F^e\left(f^{k p^m}\right),$$
	and therefore multiplication by $f^{k p^{e+m}}$ commutes with $P$. In particular, we have
	\begin{align*}
		f^{- a} P (f^{a} g) & = f^{-b - k p^{e+m}} P \left(f^{a + k p^{e+m}} g\right) \\
				    & = f^{- b - k p^{e+m}} f^{k p^{e+m}}P\left(  f^{b} g\right) \\
						  & = f^{- b} P\left(f^{b} g\right). \qedhere
	\end{align*}
\end{proof}

Given a $p$-adic integer $\alpha \in \Z_p$, a differential operator $P \in \cD_{R[f^{-1}]}$, and an element $g \in R[f^{-1}]$, we introduce the notation 
$$f^{-\alpha} P( f^\alpha g) \ceq f^{-a} P(f^{-a} g),$$
where $a \in \Z$ is chosen sufficiently close  to $\alpha \in \Z_p$ in $p$-adic metric. If $P$ has level $e$ with respect to $F$, Lemma~\ref{lemma-conj-welldef} tells us that this is independent of $a \in \Z$ as long as $\alpha - a \in p^{e+m} \Z_p$. Moreover, it is easy to see that if $P \in \cD^{(F, e)}_{R[f^{-1}]}$ is a differential operator of level $e$, then so is $f^{-\alpha} P f^\alpha$. We thus get a ring automorphism
\begin{align*}
	& \xi_{f,\alpha}  \colon \cD_{R[f^{-1}]} \lra \cD_{R[f^{-1}]}, \\ 
	& \xi_{f,\alpha} (P)  \ceq f^{-\alpha} P f^\alpha, 
\end{align*}
which one can informally think of as conjugation by $f^\alpha$. 

We will show that these automorphisms induce an automorphism of $C(\Z_p, \cD_{R[f^{-1}]})$. First, given an element $\tilde P \in C(\Z_p, \cD_{R[f^{-1}]})$, define $\xi_f( \tilde P )$ to be the function $\Z_p \to \cD_{R[f^{-1}]}$ given by
$$\left( \xi_f\left(\tilde P\right) \right) (\alpha) \ceq f^{- \alpha} \tilde P(\alpha) f^\alpha.$$

Observe that we can combine the natural filtration on $C(\Z_p, - )$ with the level filtration on $\cD_{R[f^{-1}]}$ to obtain a filtration on $C(\Z_p, \cD_{R[f^{-1}]})$ as follows: 
\begin{align*}
C\left(\Z_p, \cD_{R[f^{-1}]}\right) & = \lim_{\to e} \Fun \left(\Z / p^e, \lim_{\to i} \cD^{(F, i)}_{R[f^{-1}]}\right) \\
				 & = \lim_{\to e} \Fun \left(\Z / p^{e+m} , \cD^{(F, e)}_{R[f^{-1}]}\right), 
\end{align*}
where the second equality crucially uses that $\Z / p^e$ is a finite set.

\begin{lemma} \label{lemma-conj-filtr}
	Given $\tilde P \in \Fun (\Z / p^{e+m}, \cD^{(F, e)}_{R[f^{-1}]})$, we have $\xi_f (\tilde P) \in \Fun (\Z / p^{e+m}, \cD^{(F, e)}_{R[f^{-1}]})$. 
\end{lemma}
\begin{proof}
	It is clear that for every $\alpha \in \Z_p$ we have $(\xi_f \tilde P)(\alpha) \in \cD^{(F, e)}_{R[f^{-1}]}$, so it suffices to check that for every $\alpha, \beta \in \Z_p$ we have $(\xi_f (\tilde P))(\alpha + p^{e+m} \beta) = (\xi_f (\tilde P))(\alpha)$, which follows from Lemma~\ref{lemma-conj-welldef}.
\end{proof}

From Lemma~\ref{lemma-conj-filtr} we conclude that $\xi_f$ gives a ring automorphism 
$$\xi_f\colon C\left(\Z_p, \cD_{R[f^{-1}]}\right) \overset{\lowsim}\lra C\left(\Z_p, \cD_{R[f^{-1}]}\right).$$
Given a $C(\Z_p, \cD_{R[f^{-1}]})$-module $M$, by restricting scalars along $\xi_f$, we obtain a new module, which we denote by $\xi_{f, *} M$. Note that $\xi_{f, *} M = M$ as an abelian group. By applying this construction to the module $C(\Z_p, R[f^{-1}])$, we obtain the module that we have been looking for. 

\begin{definition}
	Let $(V, \fm, \KK, F)$ be as in Setup~\ref{setup-V}, let $R$ be a finite type $V$-algebra with a compatible lift of Frobenius $F\colon R \to R$, and let $f \in R$ be a nonzerodivisor. The module $C(\Z_p, R[f^{-1}]) \bs{f^s}$ is given by
	$$C\left(\Z_p, R\left[f^{-1}\right]\right) \bs{f^s} \ceq \xi_{f, *} C\left(\Z_p, R\left[f^{-1}\right]\right).$$
	An element $\tilde g \in C(\Z_p, R[f^{-1}])$ will be denoted by $\tilde g \bs{f^s}$ when we want to emphasize that we think of $\tilde g$ as an element of $C(\Z_p, R[f^{-1}]) \bs{f^s}$.
\end{definition}

Recall that we have a $V$-algebra isomorphism
$$
	C\left(\Z_p, \cD_R\right) \overset{\Delta}{\underset{\highsim}\lra} \left(\cD_{R[t]}\right)_0,
$$
as described in Section~\ref{subscn-alg-cts-function}. Given a $C(\Z_p, \cD_R)$-module $M$ and a $(\cD_{R[t]})_0$-module $N$, we say that an isomorphism $\alpha\colon M \xrightarrow{\lowsim} N$ is $\Delta$-semilinear whenever we have
$$\alpha(\phi \cdot u) = \Delta(\phi) \cdot \alpha(u)$$
for all $\phi \in C(\Z_p, \cD_R)$ and $u \in M$. Informally speaking, $\alpha$ exchanges the $C(\Z_p, \cD_R)$-action of $M$ with the $(\cD_{R[t]})_0$-action of $N$ in a way that is compatible with $\Delta$.

Our next goal is to show that we have a natural $\Delta$-semilinear isomorphism $C(\Z_p, R[f^{-1}]) \bs{f^s} \cong H_f$. We begin by obtaining a concrete description of the module $C(\Z_p, R[f^{-1}]) \bs{f^s}$. First observe that we have
\begin{align*}
	C\left(\Z_p, R\left[f^{-1}\right]\right) & = \lim_{\to e} \Fun \left(\Z / p^{m + e} , R\left[f^{-1}\right]\right) \bs{f^s} \\
		& = \lim_{\to e} \bigoplus_{a = 0}^{{p^{m + e} - 1}} R\left[f^{-1}\right] \ \chi^{m + e}_a \bs{f^s}.
\end{align*}

As discussed above, we also have
\begin{align*}
	C\left(\Z_p, \cD_R\right) & = \lim_{\to e} \Fun\left(\Z / p^{e+m}, \cD^{(F, e)}_R\right) \\
		       & = \lim_{\to e} \bigoplus_{a = 0}^{p^{m + e} - 1} \cD_R^{(F, e)} \ \chi^{m + e}_a.
\end{align*}
Given an integer $e \geq 0$, an operator $P \in \cD^{(F, e)}_R$, integers $0 \leq a, b < p^{m + e}$, and an element $g \in R[f^{-1}]$, the action of $C(\Z_p, \cD_R)$ on $C(\Z_p, R[f^{-1}]) \bs{f^s}$ is determined by
\begin{equation} \label{eqn-action1}
	P \chi^{m + e}_a \cdot g \chi^{m + e}_b \bs{f^s} = 
\begin{cases}
	f^{-a} P(f^a g) \chi^{m + e}_a \bs{f^s} &\text{ if } a = b, \\
	0 &\text{ otherwise.}
\end{cases}
\end{equation}

\begin{proposition}
	Let $(V, \fm, \KK, F)$ be as in Setup~{\rm\ref{setup-V}}, let $R$ be a finite-type $V$-algebra equipped with a compatible lift of Frobenius $F\colon R \to R$, and let $f \in R$ be a nonzerodivisor. There is a $\Delta$-semilinear isomorphism
	$$C\left(\Z_p, R\left[f^{-1}\right]\right) \bs{f^s} \cong H_f$$
	which exchanges $\bs{f^s} \in C(\Z_p, R[f^{-1}]) \bs{f^s}$ with $\delta_0 \in H_f$.
\end{proposition}
\begin{proof}
	As usual, fix an integer $m > 0$ such that $\fm^{m+1} = 0$ in $V$. Recall that we have
	\begin{align*}
		\Fun \left(\Z / p^{m + e} , R\left[f^{-1}\right]\right) & = \bigoplus_{ a = 0}^{p^{m + e} - 1} R\left[f^{-1}\right] \chi^{m+e}_a, \\
		H^{m+e}_f & = \bigoplus_{a = 0}^{p^{m+e} - 1} R\left[f^{-1}\right] Q^{m+e}_a,
	\end{align*}
	and thus exchanging the bases $\chi^{m+e}_a \leftrightarrow Q^{m+e}_a$ defines an $R[f^{-1}]$-linear isomorphism 
	$$\Fun\left(\Z / p^{m+e}, R\left[f^{-1}\right]\right) \cong H^{m+e}_f.$$ 

	By (\ref{eqn-action-two}) and (\ref{eqn-action1}), this isomorphism exchanges the action of $\Fun(\Z / p^{m+e}, \cD^{(F,e)}_R)$ with the action of $\cD^{(F, e)}_R \otimes_V (\cD^{(F, m + e)}_{V[t]})_0$. Moreover, by (\ref{eqn-chi-trans}) and (\ref{eqn-Q-trans}), these isomorphisms are compatible in $e \geq 0$ and therefore define an isomorphism between limits, which gives the required isomorphism. 
\end{proof}

\begin{corollary} \label{cor-Nf-C-desc}
	Let $N_f$ is the $(\cD_{R[t]})_0$-module defined in Section~{\rm\ref{subscn-malgrange-module-b-function}}. There is a $\Delta$-semilinear isomorphism
	$$\frac{C\left(\Z_p, \cD_R\right) \cdot \bs{f^s}}{ C\left(\Z_p, \cD_R\right) \cdot f \bs{f^s}} \cong N_f.$$
\end{corollary}

\begin{corollary} \label{cor-bfunction-2}
	The $b$-function $B_f$ of $f$ is the ideal of\, $C(\Z_p, V)$ given as the annihilator 
	$$B_f = \Ann \left( C\left(\Z_p, \cD_R\right) \cdot \bs{f^s} \ \big/ \ C\left(\Z_p, \cD_R\right) \cdot f \bs{f^s} \right) \sq C(\Z_p, V).$$
\end{corollary}

\subsection{The modules \texorpdfstring{$\boldsymbol{R[f^{-1}] \bs{f^\alpha}}$}{R[f-1]fa}} \label{subscn-Rf-alpha}

Let $(V, \fm, \KK, F)$ be as in Setup~\ref{setup-V}, let $R$ be a $V$-algebra, and let $f \in R$ be a nonzerodivisor. Our next goal is to leverage the alternative description of the $C(\Z_p, V)$-module $N_f$ obtained in Corollary~\ref{cor-Nf-C-desc} to give new characterizations of the Bernstein--Sato roots of $f$. This is attained in Theorem~\ref{thm-BSR-equivalences}. 

Recall that a $p$-adic integer $\alpha \in \Z_p$ is a Bernstein--Sato root of $f$ whenever the stalk of $N_f$ at $\alpha$ is nonzero. Corollary~\ref{cor-Nf-C-desc} suggests that, to get a handle on the Bernstein--Sato roots, it would be useful to understand the stalks of the module $C(\Z_p, R[f^{-1}]) \bs{f^s}$; this is the goal of the current subsection. 

We start by recalling that $C(\Z_p, R[f^{-1}) \bs{f^s}$ is naturally a $C(\Z_p, V)$-module, and therefore it gives rise to a quasicoherent sheaf on $\Spec ( C(\Z_p, V)) \cong \Z_p$ (see Section~\ref{subscn-alg-cts-function}).

\begin{definition}
     Let $(V, \fm, \KK, F)$ be as in Setup~\ref{setup-V}, let $R$ be a $V$-algebra, and let $f \in R$ be a nonzerodivisor. Given a $p$-adic integer $\alpha \in \Z_p$, we let $R[f^{-1}] \bs{f^\alpha}$ be the stalk of $C(\Z_p, R[f^{-1}]) \bs{f^s}$ at $\alpha$; that is, 
     $$R[f^{-1}] \bs{f^\alpha} \ceq \left(C\left(\Z_p, R\left[f^{-1}\right]\right) \bs{f^s} \right)_\alpha.$$
     An element $g \in R[f^{-1}]$ will be denoted by $g \bs{f^\alpha}$ if we want to emphasize that we view $g$ as an element of $R[f^{-1}] \bs{f^\alpha}$. 
\end{definition}

By Proposition~\ref{prop-C-stalk-description} there is a surjective homomorphism 
\begin{align*}
	C\left(\Z_p, R\left[f^{-1}\right]\right) \bs{f^s} & \overset{e_\alpha}\lra  R\left[f^{-1}\right] \bs{f^\alpha} \\
	\tilde g \bs{f^s} & \longmapsto \tilde g(\alpha) \bs{f^\alpha}
\end{align*}
given by evaluation at $\alpha \in \Z_p$, which gives
$$\frac{ C\left(\Z_p, R\left[f^{-1}\right]\right) \bs{f^s} }{ (0: \alpha) C\left(\Z_p, R\left[f^{-1}\right]\right) \bs{f^s} } \cong R[f^{-1}] \bs{f^\alpha} .$$

Since $C(\Z_p, R[f^{-1}]) \bs{f^s}$ is a $C(\Z_p, \cD_{R[f^{-1}]})$-module, $R[f^{-1}] \bs{f^\alpha}$ is naturally a $\cD_{R[f^{-1}]}$-module. This structure is given as follows: for $P \in \cD_{R[f^{-1}]}$ and $g \in R[f^{-1}]$, we have
$$P \cdot\left(g \bs{f^\alpha}\right) = f^{-\alpha} P (f^\alpha g) \bs{f^\alpha}$$
(see Lemma~\ref{lemma-conj-welldef} and the discussion below it). Recall that if $e, m \geq 0$ are chosen such that $\fm^{m+1} = 0$ in $V$ and $P \in \cD^{(F, e)}_R$, then this means that $P \cdot (g \bs{f^\alpha}) = f^{-a} P (f^a g) \bs{f^\alpha}$, where $a \in \Z$ is such that $a \equiv \alpha \mod p^{e+m} \Z_p$.

Note that, in the case where $V$ is a field of positive characteristic and $\alpha \in \Z_{(p)}$ is rational, the modules $R[f^{-1}] \bs{f^\alpha}$ appear in the work of Blickle, \Mustata, and Smith \cite{BMSm-hyp} under the name $R_f e_{- \alpha}$.

\begin{remark} \label{rmk-Rf-compat}
	Let $(V, \fm, \KK, F)$ be as in Setup~\ref{setup-V}, let $R$ be a $V$-algebra of finite type with a compatible lift of Frobenius $F\colon R \to R$, and let $f \in R$ be a nonzerodivisor. Let $k \in \Z$ be an integer and $\alpha \in \Z_p$ be a $p$-adic integer. 
	\begin{enuroman}
	\item The assignment $[g \bs{f^{\alpha + k}} \mapsto g f^k \bs{f^\alpha}]$ gives a $\cD_{R[f^{-1}]}$-module isomorphism
	$$
		R\left[f^{-1}\right] \bs{f^{\alpha + k}} \overset{\lowsim}\lra R\left[f^{-1}\right] \bs{f^\alpha}.
        $$
	\item Suppose $k \geq 0$. Then the assignment $[g \bs{(f^k)^{\alpha}} \mapsto g \bs{f^{k \alpha}}]$ gives a $\cD_{R[f^{-1}]}$-linear isomorphism
	$$
		R\left[f^{-1}\right] \bs{(f^k)^{\alpha}} \overset{\lowsim}\lra R\left[f^{-1}\right] \bs{f^{k \alpha}}.
	$$
	\end{enuroman}
\end{remark}

\subsection{The \texorpdfstring{$\boldsymbol{\nu}$}{nu}-invariants} \label{subscn-nu-invt}

In this subsection we provide generalizations of the $\nu$-invariants introduced by \Mustata, Takagi, and Watanabe \cite{MTW}, which we will subsequently use to provide a useful characterization of Bernstein--Sato roots. Note that, when working over a field of positive characteristic, this characterization already appears in \cite{QG19,JNBQG}.

We slightly deviate from previous work by working in the module $R[f^{-1}]$ instead of $R$. This means that we allow our $\nu$-invariants to be negative and that, when constructing the invariants $\nu^J_f(F, p^e)$ below, we must allow for $J$ to be an $R$-submodule of $R[f^{-1}]$ satisfying $J[f^{-1}] = R[f^{-1}]$, instead of an ideal of $R$ satisfying $f \in \sqrt J$. By the translation property of $\nu$-invariants given in Proposition~\ref{prop-nu-basic-props}\eqref{pnbp-2}, these changes are purely cosmetic and do not affect the rest of the theory. 

Let $(V, \fm, \KK, F)$ be as in Setup~\ref{setup-V}, let $R$ be a $V$-algebra of finite type equipped with a lift of Frobenius $F\colon R \to R$, and let $f \in R$ be a nonzerodivisor. In this subsection we will assume that $R$ is smooth, so that Frobenius descent holds as given in Corollary~\ref{cor-frob-ideal-corresp}. 

Fix an integer $e \geq 0$. Under Frobenius descent (Corollary~\ref{cor-frob-ideal-corresp}), the chain of $\cD^{(F, e)}_R$-submodules of $R[f^{-1}]$ given by
$$\cdots \supseteq \cD_R^{(F, e)} \cdot f^{-2} \supseteq \cD_R^{(F, e)} \cdot f^{-1} \supseteq \cD_R^{(F, e)} \cdot 1 \supseteq \cD_R^{(F, e)} \cdot f^1 \supseteq \cD_R^{(F, e)} \cdot f^2 \supseteq \cdots$$
is in correspondence with the chain of $R$-submodules of $R[f^{-1}]$ given by
$$\cdots \supseteq \cC_R^{(F, e)} \cdot f^{-2} \supseteq \cC_R^{(F, e)} \cdot f^{-1} \supseteq \cC_R^{(F, e)} \cdot 1 \supseteq \cC_R^{(F, e)} \cdot f^1 \supseteq \cC_R^{(F, e)} \cdot f^2 \supseteq \cdots.$$
In particular, given an integer $n \in \Z$, the following are equivalent:
\begin{enualph}
\item\label{conda} We have $\cD^{(F, e)}_R \cdot f^n \neq \cD^{(F, e)}_R \cdot f^{n+1}$.
\item\label{condb} We have $\cC^{(F, e)}_R \cdot f^n \neq \cC^{(F, e)}_R \cdot f^{n+1}$.
\end{enualph}

\begin{definition}
	Let $(V, \fm, \KK, F)$ be as in Setup~{\rm\ref{setup-V}}, and let $R$ be a smooth $V$-algebra equipped with a compatible lift of Frobenius $F\colon R \to R$. Let $f \in R$ be a nonzerodivisor and $e \geq 0$ be an integer. We say that an integer $n \in \Z$ is a $\nu$-invariant of level $e$ for $f$ with respect to $F$ whenever the equivalent conditions~\eqref{conda} and~\eqref{condb} above hold. We denote by $\nu^\bullet_f(F, p^e)$ the set of all such $\nu$-invariants.
\end{definition}

Let $J \sq R[f^{-1}]$ be an $R$-submodule, and let $e \geq 0$ be an integer. By the smoothness of $R$, the functor $F^{e*}$ is exact (see Lemma~\ref{lemma-FeR-fgp}). The natural map $F^{e*}(J) \to F^{e*}(R[f^{-1}]) \cong R[f^{-1}]$ is therefore injective, and, by an abuse of notation, we identify $F^{e*}(J)$ with its image in $R[f^{-1}]$. Whenever $J$ is such that $J[f^{-1}] = R[f^{-1}]$, we let 
$$\nu^J_f(F, p^e) \ceq \max \left\{ n \geq 0 \ | \ f^n \notin F^{e*}(J) \right\}.$$
Note that the maximum is indeed attained, since $F^{e*}(J) [f^{-1}] = F^{e*}(J[f^{-1}]) = R[f^{-1}]$, and therefore we have $f^n \in F^{e*}(J)$ for some $n \gg 0$. 

\begin{proposition}\label{prop315}
	Let $(V, \fm, \KK, F)$ be as in Setup~{\rm\ref{setup-V}}, let $R$ be a smooth $V$-algebra equipped with a lift of Frobenius $F\colon R \to R$, let $f \in R$ be a nonzerodivisor, and let $e \geq 0$ be an integer.
	\begin{enuroman}
	\item\label{prop315-1} Given an $R$-submodule $J \sq R[f^{-1}]$ with $J[f^{-1}] = R[f^{-1}]$, the integer $\nu^J_f(F, p^e)$ is a $\nu$-invariant of level $e$ for $f$ with respect to $F$; that is, we have $\nu^J_f(F, p^e) \in \nu^\bullet_f(F, p^e)$. 

	\item\label{prop315-2} All such $\nu$-invariants arise this way; that is,
		$$\nu^\bullet_f(F, p^e) = \left\{ \nu^J_f(F, p^e) \ \middle\vert \ J \sq R\left[f^{-1}\right] \text{ an $R$-submodule with } J\left[f^{-1}\right] = R\left[f^{-1}\right] \right\}.$$
	\end{enuroman}
\end{proposition}

\begin{proof}
	Let $J \sq R[f^{-1}]$ be an $R$-submodule. Recall that $F^{e*}(J)$ is a $\cD^{(F, e)}_R$-submodule of $R[f^{-1}]$ and therefore, given an integer $n \in \Z$, we have $f^n \notin F^{e*}(J)$ if and only if $\cD^{(F, e)}_R \cdot f^n \not \sq F^{e*}(J)$. In particular, if $J[f^{-1}] = R[f^{-1}]$, then for $\nu \ceq \nu^J_f(F, p^e)$ we have $\cD^{(F, e)}_R \cdot f^\nu \not \sq F^{e*}(J)$ while $\cD^{(F, e)}_R \cdot f^{\nu + 1} \sq F^{e*}(J)$, which gives the inequality $\cD^{(F, e)}_R \cdot f^\nu \neq \cD^{(F, e)}_R \cdot f^{\nu + 1}$, proving~\eqref{prop315-1}. 

	For part~\eqref{prop315-2}, suppose we have some $\nu \in \nu^\bullet_f(F, p^e)$. This gives $\cC^{(F, e)}_R \cdot f^\nu \neq \cC^{(F, e)}_R \cdot f^{\nu + 1}$. Next, we let $J \ceq \cC^{(F, e)}_R \cdot f^{\nu + 1}$ and observe that, by Frobenius descent, we know that $F^{e*}(J) = \cD^{(F, e)}_R \cdot f^{\nu + 1}$. Since $\nu \in \nu^\bullet_f(F, p^e)$ by assumption, we have $f^\nu \notin F^{e*}(J)$; on the other hand, it is clear that $f^{\nu + 1} \in F^{e*}(J)$, and therefore $\nu = \nu^J_f(F, p^e)$. 
\end{proof}

\begin{proposition} \label{prop-nu-basic-props}
	Let $(V, \fm, \KK, F)$ be as in Setup~{\rm\ref{setup-V}}, let $R$ be a smooth $V$-algebra equipped with a lift of Frobenius $F\colon R \to R$, and let $f \in R$ be a nonzerodivisor. 
	\begin{enuroman}
	\item\label{pnbp-1} The $\nu$-invariants for $f$ with respect to $F$ form a descending chain
		$$\nu^\bullet_f(F, p^0) \spq \nu^\bullet_f(F, p^1) \spq \nu^\bullet_f(F, p^2) \spq  \cdots. $$
	\item\label{pnbp-2} Fix an integer $e \geq 0$, and let $m \geq 0$ be such that $\fm^{m+1} = 0$ in $V$. Then the set $\nu^\bullet_f(F, p^e)$ is invariant under translation by integer multiples of $p^{m + e}$; that is,
	$$\nu^\bullet_f(F, p^e) + \Z p^{m+e}  = \nu^\bullet_f(F, p^e).$$

	\item\label{pnbp-3} For every integer $e \geq 0$, the $\nu$-invariants of level $e$ with respect to $F$ for $f$ agree with those of level $e + 1$ for $F(f)$; that is,
	$$\nu^\bullet_f(F, p^e) = \nu^\bullet_{F(f)} (F, p^{e+1}).$$
	\end{enuroman}
\end{proposition}
\begin{proof}
	Part~\eqref{pnbp-1} follows from the inclusions $\cD^{(F, 0)}_R \sq \cD^{(F, 1)}_R \sq \cD^{(F, 2)}_R \sq \cdots$. 

	For part~\eqref{pnbp-2} it suffices to show that, given an integer $k \in \Z$ and a $\nu$-invariant $n \in \nu^\bullet_f(F, p^e)$, we have $n + k p^{e+m} \in \nu^\bullet_f(F, p^e)$. Since $n$ is a $\nu$-invariant, we have $\cD^{(F, e)}_R \cdot f^n \neq \cD^{(F, e)}_R \cdot f^{n+1}$. Also recall  that we have $f^{p^{e+m}} = F^e(f^{p^m})$ and therefore $f^{p^{e+m}}$ commutes with all operators in $\cD^{(F, e)}_R$. Combining these observations (and using the fact that $f$ is a nonzerodivisor), we get
	\begin{align*}
		\cD^{(F, e)}_R \cdot f^{n + k p^{e+m}} = f^{k p^{e+m}} \cD^{(F, e)}_R \cdot f^n \neq f^{k p^{e+m}} \cD^{(F, e)}_R \cdot f^{n+1} = \cD^{(F, e)}_R \cdot f^{n + kp^{e+m} + 1},
	\end{align*}
	which gives $n + k p^{e+m} \in \nu^\bullet_f(F, p^e)$ as required. 

	Part~\eqref{pnbp-3} follows from a direct application of Corollary~\ref{cor-frob-cartier-inverse}.
\end{proof}

\subsection{Alternative characterizations of Bernstein--Sato roots}

Let $(V, \fm, \KK, F)$ be as in Setup~\ref{setup-V}, let $R$ be a $V$-algebra equipped with a compatible lift of Frobenius $F\colon R \to R$, and let $f \in R$ be a nonzerodivisor. In this subsection we provide two alternative characterizations of Bernstein--Sato roots of $f$. One of them uses the $\cD_R$-modules $R[f^{-1}] \bs{f^\alpha}$ introduced in Section~\ref{subscn-Rf-alpha}, while the other uses the $\nu$-invariants introduced in Section~\ref{subscn-nu-invt}. 

Let $N_f$ be the $C(\Z_p, V)$-module given in Section~\ref{subscn-malgrange-module-b-function} (see Corollary~\ref{cor-Nf-C-desc} for an alternative description), and recall that a $p$-adic integer $\alpha \in \Z_p$ is a Bernstein--Sato root of $f$ whenever the stalk $(N_f)_\alpha$ is nonzero. 

\begin{theorem} \label{thm-BSR-equivalences}
	Let $(V, \fm, \KK, F)$ be as in Setup~{\rm\ref{setup-V}}, let $R$ be a smooth $V$-algebra equipped with a compatible lift of Frobenius $F\colon R \to R$, and let $f \in R$ be a nonzerodivisor. Fix an integer $m \geq 0$ such that $\fm^{m+1} = 0$ in $V$. For a $p$-adic integer $\alpha \in \Z_p$, the following are equivalent:
	\begin{enualph}
	\item\label{t-BSR-e-a} The $p$-adic integer $\alpha \in \Z_p$ is a Bernstein--Sato root of $f$; that is, $(N_f)_\alpha \neq 0$. 
	\item\label{t-BSR-e-b} In the $\cD_R$-module $R[f^{-1}] \bs{f^\alpha}$, we have $\bs{f^\alpha} \notin \cD_R \cdot f \bs{f^\alpha}$.
	\item\label{t-BSR-e-c} For every integer $e \geq 0$, we have $(\alpha + p^{e+m} \Z_p) \cap \Z \sq \nu^\bullet_f(F, p^e)$. 
	\item\label{t-BSR-e-d} There is an integer sequence $(\nu_e)_{e = 0}^\infty \sq \Z$ such that $\nu_e \in \nu^\bullet_f(F, p^e)$ for every $e \geq 0$, and whose $p$-adic limit is $\alpha$. 
	\end{enualph}
\end{theorem}
\begin{proof}
By Corollary~\ref{cor-Nf-C-desc} we have an isomorphism 
$$(N_f)_\alpha \cong \frac{\left( C\left(\Z_p, \cD_R\right) \cdot \bs{f^s} \right)_\alpha}{ \left( C\left(\Z_p, \cD_R\right) \cdot f \bs{f^s} \right)_\alpha }. $$
Now recall that $R[f^{-1}] \bs{f^\alpha}$ is the stalk of $C(\Z_p, R[f^{-1}]) \bs{f^s}$ at $\alpha \in \Z_p$, from which we obtain
\begin{align*}
	\left( C\left(\Z_p, \cD_R\right) \cdot \bs{f^s} \right)_\alpha & = \left( \im \left( C\left(\Z_p, \cD_R\right) \xrightarrow{\bs{f^s}} C\left(\Z_p, R\left[f^{-1}\right]\right)  \bs{f^s} \right)	\right)_\alpha \\
		& = \im \left( C\left(\Z_p, \cD_R\right)_\alpha \xrightarrow{\bs{f^s}} \left(C\left(\Z_p, R\left[f^{-1}\right]\right) \bs{f^s}\right)_\alpha \right)\\
		& = \im \left( \cD_R \xrightarrow{\bs{f^\alpha}} R\left[f^{-1}\right] \bs{f^\alpha} \right) \\
		& = \cD_R \cdot \bs{f^\alpha}.
\end{align*}
Similarly, we obtain that $\left(C(\Z_p, \cD_R) \cdot f \bs{f^s} \right)_\alpha  = \cD_R \cdot f \bs{f^\alpha}$,
from which we conclude that 
$$\left(N_f\right)_\alpha \cong \frac{\cD_R \cdot \bs{f^\alpha}}{\cD_R \cdot f \bs{f^\alpha}},$$
which gives the equivalence between~\eqref{t-BSR-e-a} and~\eqref{t-BSR-e-b}. 

Suppose \eqref{t-BSR-e-b} is false, and fix an $e \gg 0$ such that $\bs{f^\alpha} \in \cD^{(F, e)}_R \cdot f \bs{f^\alpha}$. Fix an  $a \in \Z$ with the property that $\alpha \equiv a \mod p^{e+m} \Z_p$, for example by truncating the $p$-adic expansion of $\alpha$. We then have (see Section~\ref{subscn-Rf-alpha}) 
$$\bs{f^\alpha} \in \cD^{(F, e)}_R \cdot f \bs{f^\alpha} = f^{-a} \left(\cD^{(F, e)}_R \cdot f^{a+1}\right) \bs{f^\alpha},$$
and therefore $1 \in f^{-a}(\cD^{(F, e)}_R \cdot f^{a+1})$ or, equivalently, $f^a \in \cD^{(F, e)}_R \cdot f^{a+1}$. But this gives that $\cD^{(F, e)}_R \cdot f^a = \cD^{(F, e)}_R \cdot f^{a+1}$, and therefore $a \notin \nu^\bullet_f(F, p^e)$. This shows that~\eqref{t-BSR-e-c} implies~\eqref{t-BSR-e-b}. 

Conversely, suppose~\eqref{t-BSR-e-b} holds, fix an integer $e \geq 0$, and let $a \in (\alpha + p^{e+m} \Z_p) \cap \Z$; in other words, let $a \in \Z$ be an integer such that $a \equiv \alpha \mod p^{e+m} \Z_p$. As before, we have  $\cD^{(F, e)}_R \cdot f \bs{f^\alpha} = f^{-a}( \cD^{(F, e)}_R \cdot f^{a+1}) \bs{f^\alpha}$, and from the assumption that $\bs{f^\alpha} \notin \cD^{(F, e)}_R \cdot f \bs{f^\alpha}$, we conclude that $f^a \notin \cD^{(F, e)}_R \cdot f^{a+1}$. Therefore, we have $\cD^{(F, e)}_R \cdot f^a \neq \cD^{(F, e)}_R \cdot f^{a+1}$, and hence $a \in \nu^\bullet_f(F, p^e)$. We conclude that~\eqref{t-BSR-e-b} implies~\eqref{t-BSR-e-c}. 

That~\eqref{t-BSR-e-c} implies~\eqref{t-BSR-e-d} is straightforward: if we choose $\nu_e$ to be any integer with $\alpha \equiv \nu_e \mod p^{e+m} \Z_p$, then the $p$-adic limit of $(\nu_e)_{e = 0}^\infty$ is $\alpha$, and~\eqref{t-BSR-e-c} gives that $\nu_e \in \nu^\bullet_f(F, p^e)$. 

To finish the proof, assume that a sequence $(\nu_e)_{e = 0}^\infty$ as in~\eqref{t-BSR-e-d} exists. Since the $p$-adic limit of $\nu_e$ is $\alpha$, we may pass to a subsequence to assume that $\alpha \equiv \nu_e \mod p^{e+m} \Z_p$ (note that, by Proposition~\ref{prop-nu-basic-props}\eqref{pnbp-1}, the property $\nu_e \in \nu^\bullet_f(F, p^e)$ is preserved under passage to a subsequence). From the fact that $\alpha \equiv \nu_e \mod p^{e+m} \Z_p$ we conclude that $\alpha + p^{e+m} \Z_p = \nu_e + p^{e+m} \Z_p$ for every $e \geq 0$. On the other hand, since the $\nu_e$ are integers, we have $(\nu_e + p^{e+m} \Z_p) \cap \Z = \nu_e + p^{e+m} \Z$. Putting these together, we obtain
$$\left(\alpha + p^{e+m} \Z_p\right) \cap \Z = \left(\nu_e + p^{e+m} \Z_p\right) \cap \Z = \nu_e + p^{e+m} \Z.$$
Finally, by Proposition~\ref{prop-nu-basic-props}\eqref{pnbp-2} we conclude that $\nu_e + p^{e+m} \Z \sq \nu^\bullet_f(F, p^e)$. We conclude that~\eqref{t-BSR-e-d} implies~\eqref{t-BSR-e-c}, which completes the proof. 
\end{proof}

\begin{corollary} \label{cor-truncation-sequence}
	Suppose $\alpha \in \Z_p$ is a Bernstein--Sato root of $f$, and let $\alpha = \alpha_0 + p \alpha_1 + p^2 \alpha_2 + \cdots$ be its $p$-adic expansion. For every $e \geq 0$, let $\alpha_{< e+m}$ be the right truncation
	$$\alpha_{< e+m} \ceq \alpha_0 + p \alpha_1 + p^2 \alpha_2 + \cdots + p^{e+m-1} \alpha_{e + m - 1}.$$
	Then we have $\alpha_{< e+m} \in \nu^\bullet_f(F, p^e)$ for every $e \geq 0$.
\end{corollary}
\begin{proof}
	Clearly we have $\alpha_{< e + m} \equiv \alpha \mod p^{e+m} \Z_p$, or, in other words, we have $\alpha_{< e+ m} \in \alpha + p^{e+m} \Z_p$. The statement then follows from Theorem~\ref{thm-BSR-equivalences}\eqref{t-BSR-e-c}.
\end{proof}

\begin{corollary} \label{cor-BSR-cd-F-indep}
	The collection of $p$-adic integers $\alpha$ satisfying condition~\eqref{t-BSR-e-c} $($resp.~\eqref{t-BSR-e-d}$)$ of Theorem~{\rm\ref{thm-BSR-equivalences}} is independent of the choice of\, $F$.
\end{corollary}
\begin{proof} 
	This collection is the collection of Bernstein--Sato roots of $f$, and the definition of Bernstein--Sato roots makes no reference to $F$.
\end{proof}

Despite what Corollaries~\ref{cor-truncation-sequence} and~\ref{cor-BSR-cd-F-indep} might suggest, it is not true that the $\nu$-invariants $\nu^\bullet_f(F, p^e)$ are independent of the choice of lift of Frobenius $F$. As an example, let $p = 2$ and $R = (\Z / 4)[x,y]$ with the lifts of Frobenius $F_i\colon R \to R$ given by $F_1(x) = x^2$, $F_1(y) = y^2$, $F_2(x) = x^2 + 2y(x+y)$, and $F_2(y) = y^2$. By Proposition~\ref{prop-nu-basic-props}\eqref{pnbp-2}, for an arbitrary nonzerodivisor $f \in R$, the $\nu$-invariants $\nu^\bullet_f(F_i, p^1)$ of level one are uniquely determined by $\nu^\bullet_f(F_i, p^1) \cap [0, 3]$, which we give for $f = x$. On the one hand, one has $\nu^\bullet_x(F_1, p^1) \cap [0,3] = \{1,3\}$. On the other hand, note that $F_2$ is engineered so that the change of coordinates $[(x,y) \mapsto (x+y, y)]$ exchanges $F_2$ with $F_1$. One can easily compute $\nu^\bullet_{x+y}(F_1, p^1) = \{1,2,3\}$  and therefore $\nu^\bullet_x(F_2, p^1) = \nu^\bullet_{x+y}(F_1, p^1) = \{1,2,3\}$.

\begin{corollary} \label{cor-BSR-f-Ff}
	The polynomials $f$ and $F(f)$ have the same Bernstein--Sato roots.
\end{corollary}
\begin{proof}
	This follows from description~\eqref{t-BSR-e-d} of  Theorem~\ref{thm-BSR-equivalences} together with Proposition~\ref{prop-nu-basic-props}\eqref{pnbp-3}.
\end{proof}

\section{Properties of Bernstein--Sato roots}

\subsection{Finiteness and rationality}

For the whole of this subsection, we let $(V, \fm, \KK, F)$ be as in Setup~\ref{setup-V}, let $R \ceq V[x_1, \dots , x_n]$ be a polynomial ring over $V$, and let $f \in R$ be a nonzerodivisor. 

The main goal here is to prove that there are only a finite number of Bernstein--Sato roots of $f$, all of which are rational (recall that the ring $\Z_{(p)}$ of rational numbers whose denominator is not divisible by $p$ is naturally a subring of $\Z_p$, and that a $p$-adic integer is said to be rational whenever it lies in this subring). 

Note that we need to restrict to the case where $R$ is a polynomial ring over $V$ because the proof crucially uses the degree filtration. Nonetheless, we expect the same result to hold whenever $R$ is an arbitrary smooth $V$-algebra. 

As usual, we fix an integer $m \geq 0$ such that $\fm^{m+1} = 0$ in $V$. 

\begin{lemma} \label{lemma-nu-bound}
	Let $F\colon R \to R$ be the lift of Frobenius given by $F(x_i) = x_i^p$ for all $i = 1, \dots , n$. Then there is a constant $K \geq 0$ such that, for every $e \geq 0$, we have
	$$\# \left( \nu^\bullet_f(F, p^e) \cap [0, p^{e+m}) \right) \leq K.$$
\end{lemma}
\begin{proof}
	We want to find a $K$ such that, for every $e \geq 0$, the chain of ideals
	$$\cC^{(F, e)}_R \cdot f^0 \spq \cC^{(F, e)}_R \cdot f^1 \spq \cC^{(F, e)}_R \cdot f^2 \spq \cdots \spq \cC^{(F, e)}_R \cdot f^{p^{e+m}}$$
	has at most $K$ proper inclusions. 

	Suppose $f$ has degree at most $d$ (see the discussion above Proposition~\ref{prop-cartier-degrees}). By Proposition~\ref{prop-cartier-degrees} the ideal $\cC^{(F, e)}_R \cdot f^k$ is generated in degrees at most $kd / p^e$, and therefore every ideal in the chain is generated in degrees at most $d p^m$. It follows that $\cC^{(F, e)}_R \cdot f^k = \cC^{(F, e)}_R \cdot f^{k+1}$ if and only if $( \cC^{(F, e)}_R \cdot f^k ) \cap [R]_{\leq dp^m} = ( \cC^{(F, e)}_R \cdot f^{k+1} ) \cap [R]_{\leq dp^m}$, and hence the number of proper inclusions in the chain above is the same as the number of proper inclusions in the following chain of $V$-modules:
	$$\left( \cC^{(F, e)}_R \cdot f^0 \right) \cap [R]_{\leq d p^m} \spq \left( \cC^{(F, e)}_R \cdot f^1 \right) \cap [R]_{\leq d p^m} \spq \cdots \spq \left( \cC^{(F, e)}_R \cdot f^{p^{e+m}} \right) \cap [R]_{\leq dp^m}.$$
	Now note that this is a chain of $V$-submodules of $[R]_{\leq d p^m}$. Note that $[R]_{\leq dp^m}$ has finite length as a $V$-module; letting $K$ be this length, we conclude that there can be no more than $K$ proper inclusions in this chain. 
\end{proof}

\begin{theorem} \label{thm-BSR-fin}
	Let $(V, \fm, \KK, F)$ be as in Setup~{\rm\ref{setup-V}}, let $R \ceq V[x_1, \dots , x_n]$ be a polynomial ring over $V$, and let $f \in R$ be a nonzerodivisor. There are only finitely many Bernstein--Sato roots of $f$.
\end{theorem}
\begin{proof}
	Let $K$ be the bound given in Lemma~\ref{lemma-nu-bound}. We claim there can be at most $K$ Bernstein--Sato roots of~$f$.

	For a contradiction, suppose that there are at least $K+1$ distinct Bernstein--Sato roots. Call them $\alpha_1, \dots , \alpha_{K+1}$, and choose an integer $a \geq 0$ such that $(\alpha_i + p^{a+m} \Z_p) \cap (\alpha_j + p^{a+m} \Z_p) = \emptyset$ for $i \neq j$ (in metric terms, we are choosing a radius so that balls of that radius centered at the $\alpha_i$ are disjoint).

	For every $i = 1, \dots , K+1$ and every $e \geq 0$, let $\nu_{i,e}$ be the truncation of the $p$-adic expansion of $\alpha_i$ at the $(e+m)$-digit, as in Corollary~\ref{cor-truncation-sequence}. For every $i$ and $e$, we clearly have $\nu_{i, e} \in \alpha_i + p^{e+m} \Z_p$ and $0 \leq \nu_{i,e} < p^{e+m}$, and by Corollary~\ref{cor-truncation-sequence} we also know that $\nu_{i, e} \in \nu^\bullet_f(F, p^e)$. We conclude that $\nu_{1, a}, \dots , \nu_{K+1,a}$ are distinct $\nu$-invariants of level $a$, all in the interval $[0, p^{a+m})$, thus giving a contradiction to the bound from Lemma~\ref{lemma-nu-bound}.
\end{proof}

\begin{corollary} \label{cor-Nf-discrete}
	The module $N_f$ described in Section~{\rm\ref{subscn-malgrange-module-b-function}} splits as a direct sum of stalks
	$$N_f \cong \left(N_f\right)_{\alpha_1} \oplus \cdots \oplus \left(N_f\right)_{\alpha_s},$$
	where $\alpha_1, \dots , \alpha_s$ are the Bernstein--Sato roots of $f$ and the $b$-function $B_f$ of $f$ takes the form 
	$$B_f = (\fa_1 : \alpha_1) \cdot (\fa_2 : \alpha_2) \cdots (\fa_s : \alpha_s)$$
	for some ideals $\fa_1, \dots , \fa_s \sq V$.
\end{corollary}
\begin{proof}
	The statement follows by combining Theorem~\ref{thm-BSR-fin} and Proposition~\ref{prop-disc-C-module}.
\end{proof}

Having addressed the finiteness issue, we turn to rationality. The proof will proceed by first showing the statement for nonzerodivisors $f \in R$ for which there is a lift of Frobenius $F\colon R \to R$ with the property that $F(f) = f^p$; this assumption allows the characteristic-$p$ proof given in \cite{JNBQG} to go through. In particular, this proves that the rationality statement holds for $f^{p^m}$ with $f \in R$ arbitrary (see Lemma~\ref{lemma-Frob-p-power}), and the proof is finished by comparing the Bernstein--Sato roots of $f$ with those of $f^{p^m}$.

\begin{lemma} \label{lemma-BSR-g}
	Let $f \in R$ be a nonzerodivisor, and assume that $R$ admits a compatible lift of Frobenius $F\colon R \to R$ with the property that $F(f) = f^p$. 
	\begin{enuroman}
	\item\label{l-BSR-g-1} Let $e \geq 0$ be an integer, and let $n \in \Z$ be a $\nu$-invariant of level $e$ for $f$ with respect to $F$. Then for some $0 \leq i < p$ we have that $p n + i$ is a $\nu$-invariant of level $e + 1$ for $f$ with respect to $F$.
	\item\label{l-BSR-g-2}  If $\alpha \in \Z_p$ is a Bernstein--Sato root of $f$, then for some $0 \leq i < p$ we have that $p \alpha + i$ is a Bernstein--Sato root of $f$.
	\end{enuroman}
\end{lemma}
\begin{proof}
	We start with~\eqref{l-BSR-g-1}. By assumption we know that $\cC^{(F, e)}_R \cdot f^n \neq \cC^{(F,e)}_R \cdot f^{n+1}$. By applying Corollary~\ref{cor-frob-cartier-inverse} to $I = (f^n)$, and using that $F^*(f^n) = (F(f^n)) = f^{pn}$, we conclude that $\cC^{(F, e)}_R \cdot f^n = \cC^{(F, e+1)}_R \cdot f^{pn}$. Similarly, we obtain that $\cC^{(F, e)}_R \cdot f^{n+1} = \cC^{(F, e+1)}_R \cdot f^{p(n+1)}$, and we conclude that $\cC^{(F, e+1)}_R \cdot f^{pn} \neq \cC^{(F, e+1)}_R \cdot f^{p(n+1)}$. In other words, in the chain of ideals
	$$\cC^{(F, e+1)}_R \cdot f^{pn} \spq \cC^{(F, e)}_R \cdot f^{pn + 1} \spq \cdots \spq \cC^{(F, e)}_R \cdot f^{pn + p},$$
	there must be at least one proper inclusion, which gives the result.

	Let us now tackle~\eqref{l-BSR-g-2}. By Theorem~\ref{thm-BSR-equivalences} there is an integer sequence $(\nu_e)_{e = 0}^\infty \sq \Z$ with $\nu_e \in \nu^\bullet_f(F, p^e)$ whose $p$-adic limit is $\alpha$. By part~\eqref{l-BSR-g-1}, for every $e$, we can find some $0 \leq i_e < p$ with the property that $p \nu_e + i_e \in \nu^\bullet_f(F, p^{e+1})$.

	Note that $i_e$ can only take finitely many values, and hence we can find some $0 \leq i < p$ such that $i_e = i$ for infinitely many $e$. By replacing $(\nu_e)$ with a subsequence and applying Proposition~\ref{prop-nu-basic-props}\eqref{pnbp-1}, we get $p \nu_e + i \in \nu^\bullet_f(F, p^e)$. Since the $p$-adic limit of this sequence is $p \alpha + i$, we conclude that $p \alpha + i$ is a Bernstein--Sato root of $f$ by applying Theorem~\ref{thm-BSR-equivalences}.
\end{proof}

We now discuss how $\nu$-invariants and Bernstein--Sato roots behave when we take $p$-power exponents.

\begin{lemma} \label{lemma-p-power-nu-BSR}
	Let $k \geq 1$ be an integer. 
	\begin{enuroman}
	\item\label{lppnBSR-1} If $n \in \nu^\bullet_f(F, p^e)$, then $\lfloor n / k \rfloor \in \nu^\bullet_{f^k}(F, p^e)$.
	\item\label{lppnBSR-2} Let $\alpha \in \Z_p$ be a $p$-adic integer with $p$-adic expansion $\alpha = \alpha_0 + p \alpha_1 + p^2 \alpha_2 + \cdots$, and consider its left truncation $\alpha_{\geq k} \ceq \alpha_k + p \alpha_{k+1} + p^2 \alpha_{k + 2} + \cdots$. If $\alpha$ is a Bernstein--Sato root of $f$, then $\alpha_{\geq k}$ is a Bernstein--Sato root of $f^{p^k}$. 
	\end{enuroman}
\end{lemma}
\begin{proof}
	For~\eqref{lppnBSR-1} consider the chain of ideals given by
	$$\cD^{(F, e)}_R \cdot f^{k \lfloor n / k \rfloor} \spq
	\cD^{(F, e)}_R \cdot f^n \spq
	\cD^{(F, e)}_R \cdot f^{n+1} \spq 
	\cD^{(F, e)}_R \cdot f^{k(\lfloor n / k \rfloor + 1)},$$
	and observe that if the inclusion in the middle is strict, then the outer two ideals cannot be equal.

	Let us now move on to~\eqref{lppnBSR-2}. If $\alpha \in \Z_p$ is a Bernstein--Sato root of $f$, then by Corollary~\ref{cor-truncation-sequence} we have that $\alpha_{<e + m} \in \nu^\bullet_f(F, p^e)$. By part~\eqref{lppnBSR-1} we conclude that $\lfloor \alpha_{< e+ m} / p^k \rfloor \in \nu^\bullet_{f^{p^k}}(F, p^e)$. To complete the proof, we observe that $\lfloor \alpha_{< e + m} / p^k \rfloor$ gives the left truncation of $\alpha_{\geq k}$; that is,
	$$\left\lfloor \alpha_{< e + m} / p^k \right\rfloor = \left(\alpha_{\geq k}\right)_{< e + m -k}$$
	for all $e \geq m - k + 1$. In particular, the $p$-adic limit of $\lfloor \alpha_{< e + m} / p^k \rfloor$ is $\alpha_{\geq k}$, and the result follows from Theorem~\ref{thm-BSR-equivalences}.
\end{proof}

\begin{theorem} \label{thm-BSR-rat}
	Let $(V, \fm, \KK, F)$ be as in Setup~{\rm\ref{setup-V}}, let $R \ceq V[x_1, \dots , x_n]$ be a polynomial ring over $V$, and let $f \in R$ be a nonzerodivisor. All the Bernstein--Sato roots of $f$ are rational.
\end{theorem}
\begin{proof}
	Fix an $m \geq 0$ such that $\fm^{m+1} = 0$ in $V$. Recall that a $p$-adic integer is rational precisely when its $p$-adic expansion is eventually repeating; using the terminology of Lemma~\ref{lemma-p-power-nu-BSR}, this entails that $\alpha \in \Z_p$ is rational if and only if $\alpha_{\geq m}$ is rational. By Lemma~\ref{lemma-p-power-nu-BSR} it suffices to show that the Bernstein--Sato roots of $f^{p^m}$ are rational. Since $F(f^{p^m}) = f^{p^{m+1}}$ (see Lemma~\ref{lemma-Frob-p-power}), we may replace $f$ with $f^{p^m}$ and therefore assume that $F(f) = f^p$. 

	Let $\BSR(f) \sq \Z_p$ denote the set of Bernstein--Sato roots of $f$, and let $[\BSR(f)]$ denote its image under the quotient map $\Z_p \to \Z_p / \Z$. Given a $p$-adic integer $\alpha \in \Z_p$, we denote by $[\alpha]$ its image in $\Z_p / \Z$.
	
	By Lemma~\ref{lemma-BSR-g}\eqref{l-BSR-g-2} the set $[ \BSR(f) ]$ is stable under multiplication by $p$, and it is finite by Theorem~\ref{thm-BSR-fin}. We conclude that, given an $\alpha \in \BSR(f)$, there exist integers $a \geq 0$ and $b \geq 1$ such that $[p^a \alpha] = [p^{a+b} \alpha]$, and therefore $p^a \alpha = p^{a+b} \alpha + c$, where $c \in \Z$ is an integer. We conclude that $\alpha = c / p^a(1 - p^b)$, and therefore $\alpha$ is rational (since $\alpha \in \Z_p$, we can also conclude that $p^a$ must divide $c$).
\end{proof}

\subsection{Negative Bernstein--Sato roots}

Let $(V, \fm, \KK, F)$ be as in Setup~\ref{setup-V}, let $R \ceq V[x_1, \dots , x_n]$ be a polynomial ring over $V$, and let $f \in R$ be a nonzerodivisor. From Theorems~\ref{thm-BSR-fin} and~\ref{thm-BSR-rat} we know that, just like in characteristic zero, there are a finite number of Bernstein--Sato roots of $f$, all of which are rational. However, in characteristic zero, a result of Kashiwara also tells us that roots of the Bernstein--Sato polynomial are negative; see \cite{Kas76}. The analogous result is also known to hold in our setting when $V$ is a perfect field of characteristic $p$, by a result of the first author (see \cite{Bitoun2018}, restated in the introduction as Theorem~\ref{thm-bitoun-intro}). One might therefore expect that the same should hold true in our setting, but, as shown in Example~\ref{ex-pos-BSR}, this need not always be the case. Nonetheless, we can show that this holds in certain situations. 

\begin{proposition} \label{prop-neg-good-Frob}
	Let $(V, \fm, \KK, F)$ be as in Setup~{\rm\ref{setup-V}}, let $R \ceq V[x_1, \dots , x_n]$ be a polynomial ring over $V$, and let $f \in R$ be a nonzerodivisor. 
	\begin{enuroman}
	\item\label{pngF-1} All Bernstein--Sato roots of $f$ which are integers are negative.
	\item\label{pngF-2} Assume there exists a lift of Frobenius $F\colon R \to R$ such that $F(f) = f^p$. Then all the Bernstein--Sato roots of $f$ are negative.
	\end{enuroman}
\end{proposition}
\begin{proof}
	For part~\eqref{pngF-1}, we use the characterization of Bernstein--Sato roots given in Theorem~\ref{thm-BSR-equivalences}\eqref{t-BSR-e-b}. We claim that $\cD_R \cdot f^n = R$ for every integer $n \geq 0$. When $V = \KK$ is a perfect field of positive characteristic, this is true because $R$ is strongly $F$-regular: there exists some integer $e \geq 0$ such that $\cC^{(F, e)}_R \cdot f^n = R$, and therefore $\cD^{(F, e)}_R \cdot f^n = F^{e*}(R) = R$ (see \cite{Smi95} for more details). In our generality this tells us that $\cD_R \cdot f^n + \fm R = R$ (see the discussion below), and the statement follows because $\fm R$ is a nilpotent ideal. 

	Recall that, given a nonnegative integer $n \geq 0$, there is a $\cD_R$-module isomorphism $R[f^{-1}] \bs{f^n} \cong R[f^{-1}]$ given by exchanging the symbol $\bs{f^n}$ with the element $f^n \in R$. By the claim we get an isomorphism $\cD_R \cdot \bs{f^n} \cong \cD_R \cdot f^n = R$ and $\cD_R \cdot f \bs{f^n} \cong \cD_R \cdot f^{n+1} = R$, and therefore we have $\cD_R \cdot \bs{f^n} = \cD_R \cdot f \bs{f^n}$, thus proving~\eqref{pngF-1}. 

	For part~\eqref{pngF-2}, if $\alpha \in \Z_{(p), >0}$ were a positive Bernstein--Sato root of $f$, then we would have $p \alpha + i > \alpha$ for all $i = 0, 1, \dots , p-1$. By iterating Lemma~\ref{lemma-BSR-g}\eqref{l-BSR-g-2}, we would obtain an infinite number of Bernstein--Sato roots of~$f$, which would contradict Theorem~\ref{thm-BSR-fin}. 
\end{proof}

In this subsection our goal is to show that the Bernstein--Sato roots that are negative agree with the Bernstein--Sato roots of the mod-$\fm$ reduction of $f$ (these, in turn, are characterized by the $F$-jumping numbers by work of the first author; see Theorem~\ref{thm-bitoun-intro}).

To explain this precisely, we set up some notation. Throughout this subsection, given a $V$-module $M$, we will use $M_0$ to denote 
$$M_0 \ceq M / \fm M.$$
(Beware of the inconsistency with the stalk notation from Section~\ref{subscn-alg-cts-function}: even when $M$ is a $C(\Z_p, V)$-module, this is not the stalk at $0 \in \Z_p$.) Similarly, given an element $u \in M$, we denote by $u_0 \in M_0$ the class of $u$ in the quotient. With this notation,  $R_0 = \KK[x_1, \dots , x_n]$ is a polynomial ring over $\KK$, and, since $f \in R$ was assumed to be a nonzerodivisor, we know that $f_0 \in R_0$ is nonzero.
 
Choose a lift of Frobenius $F\colon R \to R$; by an abuse of notation, we will also use $F\colon R_0 \to R_0$ to denote the Frobenius on $R_0$. Recall that we have a description of $\cD_R$ as
$$\cD_R = \bigcup_{e = 0}^\infty \Hom_R( F^e_* R, F^e_* R).$$
Since we have $(F^e_* R)_0 = F^e_* (R_0)$ (see proof of Lemma~\ref{lemma-FeR-fgp}), we conclude that
$$\left(\cD_R\right)_0 = \cD_{R_0}.$$
In particular, if $M$ is a $\cD_R$-module, then $M_0$ is naturally a $\cD_{R_0}$-module.

Given a $p$-adic integer $\alpha \in \Z_p$, the natural $\cD_{R_0}$-module isomorphism $R[f^{-1}]_0 \cong R_0[f_0^{-1}]$ induces a $\cD_{R_0}$-module isomorphism
$$\left(R\left[f^{-1}\right] \bs{f^\alpha}\right)_0 \cong R_0\left[f_0^{-1}\right] \bs{f_0^\alpha}.$$

We will use the following result of Blickle, \Mustata, and Smith; for simplicity, we state it only for polynomial rings.

\begin{theorem}[\textit{cf.} {\cite[Theorem~2.11]{BMSm-hyp}}] \label{thm-BMSm-hyp}
	Let $R_0 \ceq \KK[x_1, \dots , x_n]$ be a polynomial ring over a perfect field~$\KK$, and let $f_0 \in R_0$ be a nonzero element. If $\alpha \in \Z_{(p), <0}$ is a negative rational number whose denominator is not divisible by $p$, we have
	$$\cD_{R_0} \cdot \bs{f_0^\alpha} = R_0\left[f_0^{-1}\right] \bs{f_0^\alpha}.$$
\end{theorem}

Note that the module $R_0[f_0^{-1}] \bs{f_0^\alpha}$ is denoted by $M_{-\alpha}$ in \cite{BMSm-hyp}, and that the element we denote as $\bs{f_0^\alpha} \in R_0[f_0^{-1}] \bs{f^\alpha}$ is  denoted there by $e_{-\alpha} \in M_\alpha$. We refer the reader to \cite{BMSm-hyp} for the proof of the above result.

\begin{lemma} \label{lemma-phi0-surjective}
	Let $\phi\colon M \to N$ be a homomorphism of\, $V$-modules. Then $\phi$ is surjective if and only if the induced map $\phi_0\colon M_0 \to N_0$ is surjective.
\end{lemma}
\begin{proof}
	The ``only if'' direction follows from the right exactness of the functor $(-)_0 = (-) \otimes_V \KK$. For the ``if'' direction, let $Q$ be the cokernel of $\phi$, so that $M \to N \to Q \to 0$ is exact. By right exactness we get that $M_0 \to N_0 \to Q_0 \to 0$ is exact, and hence $Q_0 = 0$. We conclude that $Q = \fm Q$, and hence $Q = \fm Q = \fm^2 Q = \cdots  = \fm^{m+1} Q = 0$.
\end{proof}

Using Lemma~\ref{lemma-phi0-surjective}, we show that Theorem~\ref{thm-BMSm-hyp} also applies in our setting.

\begin{corollary} \label{cor-BMSm-hyp-1}
	Let $(V, \fm, \KK, F)$ be as in Setup~{\rm\ref{setup-V}}, let $R \ceq V[x_1, \dots , x_n]$ be a polynomial ring over $V$, let $f \in R$ be a nonzerodivisor, and let $\alpha \in \Z_{(p), <0}$ be a negative rational number whose denominator is not divisible by $p$. We then have
	$$\cD_R \cdot \bs{f^\alpha} = R\left[f^{-1}\right] \bs{f^\alpha}.$$
\end{corollary}
\begin{proof}
	The goal is to show that the inclusion $\cD_R \cdot \bs{f^\alpha} \to R[f^{-1}] \bs{f^\alpha}$ is surjective; by Lemma~\ref{lemma-phi0-surjective} this can be checked after applying $(-)_0$. The image of the induced morphism $(\cD_R \cdot \bs{f^\alpha})_0 \to R_0 [f^{-1}_0] \bs{f^\alpha_0}$ is $\cD_{R_0} \cdot \bs{f^\alpha_0}$, and the result then follows from Theorem~\ref{thm-BMSm-hyp}. 
\end{proof}
\begin{remark} \label{rmk-mod-m-generator}
	Note that the proof of Corollary~\ref{cor-BMSm-hyp-1} works more generally to show the following. If $g \in R[f^{-1}]$ is an element whose mod-$\fm$ reduction $g_0 \in R_0[f_0^{-1}]$ satisfies $\cD_{R_0} \cdot g_0 \bs{f_0^\alpha} = R_0[f_0^{-1}] \bs{f_0^\alpha}$, then we have $\cD_R \cdot g \bs{f^\alpha} = R[f^{-1}] \bs{f^\alpha}$.
\end{remark}
\begin{theorem} \label{thm-BSR-neg}
	Let $(V, \fm, \KK, F)$ be as in Setup~{\rm\ref{setup-V}}, let $R \ceq V[x_1, \dots , x_n]$ be a polynomial ring over $V$, and let $f \in R$ be a nonzerodivisor. Let $\alpha \in \Z_{(p),<0}$ be a negative rational number whose denominator is not divisible by $p$. Then $\alpha$ is a Bernstein--Sato root of $f$ if and only if it is a Bernstein--Sato root of $f_0$.
\end{theorem}
\begin{proof}
	By Theorem~\ref{thm-BSR-equivalences} the number $\alpha$ is a Bernstein--Sato root of $f$ if and only if the natural inclusion $i\colon \cD_R \cdot f \bs{f^\alpha} \to \cD_R \cdot \bs{f^\alpha}$ is an equality (\textit{i.e.}~surjective). By Lemma~\ref{lemma-phi0-surjective} this surjectivity is equivalent to the surjectivity of the induced map $i_0\colon (\cD_R \cdot f \bs{f^\alpha})_0 \to (\cD_R \cdot \bs{f^\alpha})_0$. Similarly, $\alpha$ is a Bernstein--Sato root of $f_0$ if and only if the natural inclusion $j\colon \cD_{R_0} \cdot f_0 \bs{f_0^\alpha} \to \cD_{R_0} \cdot \bs{f_0^\alpha}$ is an equality (\textit{i.e.}~surjective). 

	Note that $\cD_{R_0} \cdot f_0 \bs{f_0^\alpha}$ is the image of $(\cD_R \cdot f \bs{f^\alpha})_0 \to R_0[f_0^{-1}] \bs{f_0^\alpha}$, and similarly $\cD_R \cdot \bs{f_0^\alpha}$ is the image of $(\cD_R \cdot \bs{f^\alpha})_0 \to R_0[f_0^{-1}] \bs{f_0^\alpha}$. This means that the relevant maps fit into a diagram as follows:
	$$\begin{tikzcd}
		\left(\cD_R \cdot f \bs{f^\alpha}\right)_0 \arrow[r, "i_0"] \arrow[d, twoheadrightarrow] & \left(\cD_R \cdot \bs{f^\alpha}\right)_0 \arrow[d, twoheadrightarrow] \\
		\cD_{R_0} \cdot f_0 \bs{f_0^\alpha} \arrow[r, "j",swap] & \cD_{R_0} \cdot \bs{f_0^\alpha}\rlap{.}
	\end{tikzcd}$$
	
	At this stage one can already see that $j$ is surjective whenever $i_0$ is surjective; this proves the ``if'' direction of the statement. To prove the ``only if'' direction, we want to show that $i_0$ is surjective whenever $j$ is surjective, and for this it suffices to show that the right vertical arrow in the above diagram is an isomorphism. But this follows by Theorem~\ref{thm-BMSm-hyp} and Corollary~\ref{cor-BMSm-hyp-1}; indeed, we have 
	\begin{align*}
		\left(\cD_R \cdot \bs{f^\alpha}\right)_0 & = \left(R\left[f^{-1}\right] \bs{f^\alpha}\right)_0 \\
					      & = R_0\left[f_0^{-1}\right] \bs{f_0^\alpha} \\
					      & = \cD_{R_0} \cdot \bs{f_0^\alpha}. \qedhere
	\end{align*}
\end{proof}
\begin{remark}
In the proof of Theorem~\ref{thm-BSR-neg}, the hypothesis on the negativity of $\alpha$ is only used to ensure that the natural map $(\cD_R \cdot \bs{f^\alpha})_0 \to \cD_{R_0} \cdot \bs{f^\alpha_0}$ is an isomorphism. 
\end{remark}

\subsection{Positive Bernstein--Sato roots}

In this subsection we show through an example that there exist positive Bernstein--Sato roots, and we show that such roots can only arise as integer translates of negative roots. We begin with our example.

\begin{example} \label{ex-pos-BSR}
	Suppose $p \neq 2$, let $R \ceq (\Z / p^2)[X,Y]$, and consider $f \ceq X^2 + pY \in R$. We let $F\colon R \to R$ be the lift of Frobenius for which $F(X) = X^p$ and $F(Y) = Y^p$. We claim that for all $e \geq 2$ we have
	$$\nu^\bullet_f(F, p^e) = \{ k p^e - 1 \ | \ k \in \Z \} \cup \left\{ \frac{ k p^e - 1}{2} \ \middle\vert \ k \in \Z \text{ odd } \right\} \cup \left\{ \frac{k p^e + 1}{2} \ \middle\vert \ k \in \Z \text{ odd} \right\},$$
	and therefore
	$$\BSR(f) = \left\{ -1, \frac{-1}{2}, \frac{1}{2} \right\}.$$
	In particular, the polynomial $f$ has a positive Bernstein--Sato root.

	To prove the claim, recall that by the translation property of $\nu$-invariants (Proposition~\ref{prop-nu-basic-props}(ii)), it suffices to compute the set $\nu^\bullet_f(F, p^e) \cap [0, p^{e+1})$. Note that, for all integers $e,n \geq 0$, we have
	\begin{align*}
		F^e_* f^n & = F^e_* \left(X^2 + pY\right)^n = F^e_* \left( X^{2n} + n p X^{2n-2} Y \right) \\
			  & = X^{\lfloor 2n / p^e \rfloor} F^e_* \left(X^{2n - p^e \lfloor 2n / p^e \rfloor}\right) + np X^{\lfloor (2n-2) / p^e \rfloor} F^e_* \left(X^{2n-2 - p^e \lfloor (2n-2) / p^e \rfloor} Y\right),
	\end{align*}
	and from Proposition~\ref{prop-cartier-comput} we conclude that
	$$\cC^{(F, e)}_R \cdot f^n = \left( X^{\lfloor 2n / p^e \rfloor}, \ n p \  X^{\lfloor (2n-2) / p^e \rfloor} \right).$$

	Therefore, for $n$ to be a $\nu$-invariant, one of the two generators above must change when $n$ is replaced by $n + 1$, and hence one must necessarily have $\lfloor 2n / p^e \rfloor \neq \lfloor 2n+2 / p^e \rfloor$ or $\lfloor (2n-2) / p^e \rfloor \neq \lfloor 2n / p^e \rfloor$. This forces $2n = k p^e + a$ for some $k \in \Z$ and $a \in \{-2, -1, 0, 1\}$. We analyze each of these cases separately:
\begin{itemize}
\item If $2n = kp^e - 2$, then $k$ must be even, say $k = 2k'$, and thus $n = k' p^e - 1$. For $e 
\geq 2$, we get $\cC^{(F, e)}_R \cdot f^n = (X^{k-1})$, while $\cC^{(F, e)}_R \cdot f^{n+1} = (X^k)$. Thus, $n = k' p^e - 1 \in \nu^\bullet_f(F, p^e)$.

\item If $2n = k p^e - 1$, then $k$ must be odd. Both $n$ and $n + 1$ are units in $\Z / p^2$, and we get $\cC^{(F, e)}_R \cdot f^n = (X^{k-1})$ while $\cC^{(F, e)}_R \cdot f^{n+1} = (X^k, p X^{k-1})$. Thus, $n = (k p^e - 1)/2 \in \nu^\bullet_f(F, p^e)$. 

\item If $2n = k p^e$, then we get $\cC^{(F, e)}_R \cdot f^n = (X^k) = \cC^{(F, e)}_R \cdot f^{n+1}$, and hence we do not get $\nu$-invariants of this form.

\item If $2n = k p^e + 1$, then $k$ must be odd, and we get $\cC^{(F, e)}_R \cdot f^n = (X^k, p X^{k-1})$ while $\cC^{(F, e)}_R \cdot f^{n+1} = (X^k)$. Therefore, $n = (kp^e + 1)/ 2 \in \nu^\bullet_f(F, p^e)$.
\end{itemize}
\end{example}

In the above example, note that the positive Bernstein--Sato root $1/2$ is an integer translate of a negative one; namely, $1/2 = -1/2 + 1$. We show that this is a general phenomenon. 

\begin{theorem} \label{thm-Z-translate}
	Let $(V, \fm, \KK, F)$ be as in Setup~{\rm\ref{setup-V}}, let $R \ceq V[x_1, \dots , x_n]$ be a polynomial ring over $V$, and let $f \in R$ be a nonzerodivisor. Every Bernstein--Sato root of $f$ is a $\Z$-translate of a negative one.
\end{theorem}
\begin{proof}
We show that if $\alpha \in \Z_{(p)} \cap [-1, 0)$ is not a Bernstein--Sato root of $f$, then no integer translate can be a Bernstein--Sato root. Let $R_0 \ceq R / \fm R$ and $f_0 \in R_0$ be the mod-$\fm$ reduction of $f$. Since $\alpha$ is not a Bernstein--Sato root of $f$, we know that it is not a Bernstein--Sato root of $f_0$.

Since the Bernstein--Sato roots of $f_0$ are contained in $[-1, 0)$ by \cite{Bitoun2018} (see Theorem~\ref{thm-bitoun-intro}), we know that no integer translate of $\alpha$ is a Bernstein--Sato root of $f_0$. This (combined with Theorem~\ref{thm-BSR-equivalences} and Remark~\ref{rmk-Rf-compat}) implies that $\cD_{R_0} \cdot f^k \bs{f_0^\alpha} = R_0[f_0^{-1}] \bs{f_0^\alpha}$ for all $k \in \Z$. By Remark~\ref{rmk-mod-m-generator} we conclude that $\cD_R \cdot f^k \bs{f^\alpha} = R[f^{-1}] \bs{f^\alpha}$, and therefore no integer translate of $\alpha$ can be a Bernstein--Sato root of $f$.
\end{proof}

\begin{corollary} \label{cor-BSR-Z-lift-indep}
Let $R_0 \ceq R / \fm R$, and let $f_0 \in R_0$ be the mod-$\fm$ reduction of $f$. We then have
$$\BSR(f) + \Z = \BSR(f_0) + \Z.$$
\end{corollary}

\subsection{Strength of a Bernstein--Sato root and mod-\texorpdfstring{$\boldsymbol{p}$}{p} reduction}

By considering the notion of multiplicity for a Bernstein--Sato root, we are lead to consider the following quantity (since eventually we observe that this does not provide a good analogue of multiplicity, we give it its own name).

\begin{definition}
	Let $(V, \fm, \KK, F)$ be as in Setup~\ref{setup-V}, let $R \ceq V[x_1, \dots, x_n]$ be a polynomial ring over $V$, and let $f \in R$ be a nonzerodivisor. Given a $p$-adic integer $\alpha \in \Z_p$, the strength of $\alpha$ in $f$ is given by
	\begin{align*}
		\str(\alpha, f) & \ceq \min \left\{ t \geq 0 \ \middle\vert \ (\fm : \alpha)^t \cdot \left(N_f\right)_\alpha = 0 \right\}. 
	\end{align*}
\end{definition}

If we fix an $m \geq 0$ such that $\fm^{m+1} = 0$ in $V$, then $\str(\alpha, f)$ is an integer with $0 \leq \str(\alpha, f) \leq m + 1$, and $\alpha$ is a Bernstein--Sato root of $f$ if and only if $\str(\alpha, f) \geq 1$. We now give a few alternative descriptions of strength.

First recall that we have a natural identification $(N_f)_\alpha \cong N_f / (0 : \alpha) N_f$, that evaluation at $\alpha \in \Z_p$ defines a surjective map $C(\Z_p, V) \to V$, and that, given an ideal $\fa \sq V$, its preimage along this map $C(\Z_p, V) \to V$ is given by $(\fa : \alpha)$. With this in mind, we observe that $(\fm : \alpha)^t \cdot (N_f)_\alpha = \fm^t \cdot (N_f)_\alpha$ for every $t \geq 0$, and therefore
$$\str(\alpha, f) = \min \left\{ t \geq 0 \ \middle\vert \ \fm^t \cdot \left(N_f\right)_\alpha = 0 \right\}.$$

Suppose that $\alpha_1, \dots , \alpha_s \in \Z_p$ are the Bernstein--Sato roots of $f$. Recall that the $b$-function of $f$ takes the form
$$B_f = (\fa_1 : \alpha_1) \cdot (\fa_2 : \alpha_2) \cdots (\fa_s : \alpha_s) \sq C(\Z_p, V),$$
where $\fa_i \ceq \Ann_V( (N_f)_{\alpha_i})$. We conclude that, for any $i = 1, \dots , s$, we have
$$\str(\alpha_i, f) = \min \left\{t \geq 0 \ \middle\lvert \ \fm^t \sq \fa_i \right\}.$$
Let us describe what this entails in a fairly common situation.

\begin{example}
	Suppose that $(\mathcal W, \fm)$ is a discrete valuation ring of residue characteristic $p > 0$ and that $V = \mathcal W / \fm^{m+1}$. Then all ideals of $V$ take the form $\fm^t$ for some $0 \leq t \leq m + 1$. If $\Ann_V( (N_f)_\alpha) = \fm^t$ for some $0 \leq t \leq m + 1$, then $\str(\alpha, f) = t$. 
\end{example}

Finally, by Corollary~\ref{cor-Nf-C-desc} we have an isomorphism 
$$\left(N_f\right)_\alpha \cong \frac{\cD_R \cdot \bs{f^\alpha}}{\cD_R \cdot f \bs{f^\alpha}},$$
and, since $V$ is central in $\cD_R$, we get that $\fm^k (N_f)_\alpha = 0$ if and only if $\fm^k \bs{f^\alpha} \sq \cD_R \cdot f \bs{f^\alpha}$ in the module $R[f^{-1}] \bs{f^\alpha}$. We conclude that
$$\str(\alpha, f) = \min \left\{ t \geq 0 \ \middle\vert \ \fm^t \bs{f^\alpha} \sq \cD_R \cdot f \bs{f^\alpha} \right\}.$$
Using this description, we can give a simple example.

\begin{example} \label{ex-x1-strength}
	Suppose that $f$ is one of the variables of $R$, say $f = x_1$, and that $\alpha = -1$. Let $m \geq 0$ be the smallest nonnegative integer for which $\fm^{m+1} = 0$ in $V$. Recall that we have a natural $\cD_R$-module isomorphism $R[x_1^{-1}] \bs{x_1^{-1}} \cong R[x_1^{-1}]$ that exchanges $\bs{x_1^{-1}}$ with $x_1^{-1}$ and identifies $\cD_R \cdot x_1 \bs{x_1^{-1}} \cong R$. Since $\fm^t x_1^{-1} \in R$ if and only if $t \geq m + 1$, we conclude that $\str(-1, x_1) = m + 1$.
\end{example}

Finally, we show that replacing $V$ with a quotient can only make the strength decrease. More precisely, let $(V, \fm, \KK, F)$ be as in Setup~\ref{setup-V}, let $\fb \sq V$ is a proper ideal with $F(\fb) \sq \fb$, and let $\overline V \ceq V / \fb$. Then $\overline V$ inherits the lift of Frobenius from $V$, and therefore still satisfies the conditions of Setup~\ref{setup-V}. 

Let $R \ceq V[x_1, \dots , x_n]$ be a polynomial ring over $V$, and let $\overline R \ceq \overline V [x_1, \dots , x_n]$ be a polynomial ring over $\overline V$. Let $f \in R$ be a nonzerodivisor, and let $\overline f \ceq [f \mod \fb R] \in \overline R$ denote the image of $f$. Let $\alpha \in \Z_p$ be a $p$-adic integer. Recall that there is a surjective ring homomorphism $\cD_R \twoheadrightarrow \cD_{\bar R}$ which identifies $\cD_{\bar R} \cong \cD_R / \fb \cD_R$ (since $V$ is central in $\cD_R$, $\fb \cD_R$ is a two-sided ideal). Therefore, there is a natural correspondence between $\cD_{\bar R}$-modules and $\cD_R$-modules annihilated by $\fb$. 

\begin{proposition} \label{prop-strength-red}
	With the notation as above, we have
	$$\str(f, \alpha) \geq \str(\overline f, \alpha).$$
\end{proposition}
\begin{proof}
	The natural isomorphism $\overline R[ \bar f^{-1}] \cong R[f^{-1}] / \fb R[f^{-1}]$ gives rise to a $\cD_R$-module isomorphism
	$$\overline R\left[\bar f^{-1}\right] \bs{\bar f^\alpha}  \cong R\left[f^{-1}\right] \bs{f^\alpha} \big/ \fb R\left[f^{-1}\right] \bs{f^\alpha}$$
	which exchanges $\bs{ \bar f^\alpha}$ with the class of $\bs{f^\alpha}$. As a consequence, we have surjective homomorphisms $\cD_R \cdot \bs{f^\alpha} \twoheadrightarrow \cD_{\bar R} \cdot \bs{\bar f^\alpha}$ and $\cD_R \cdot f \bs{f^\alpha} \twoheadrightarrow \cD_{\bar R} \cdot \bar f \bs{\bar f^\alpha}$. We conclude that $(N_{\bar f})_\alpha$ is a quotient of $(N_f)_\alpha$, and the result follows.
\end{proof}

We now discuss how the strengths of a given Bernstein--Sato root can detect whether such a root arises in characteristic zero. We begin by placing ourselves in a mod-$p$ reduction situation as follows.

Let $R_\C \ceq \C[x_1, \dots , x_n]$ be a polynomial ring over $\C$, let $f \in R_\C$ be a nonzero nonunit polynomial, and let $\cD_\C$ denote the ring of $\C$-linear differential operators of $R_\C$.

We let $b_f(s) \in \C[s]$ denote the Bernstein--Sato polynomial of $f$; recall that it satisfies a functional equation of the form
\begin{equation} \label{eqn-bfunction}
b_f(s) f^s = P(s) \cdot f^{s+1}
\end{equation}
for some differential operator $P(s) \in \cD_\C [s]$. 

Given a finitely generated $\Z$-subalgebra $\mathcal V \sq \C$, we let $R_{\mathcal V} \ceq \mathcal V[x_1, \dots , x_n]$ be the subring of $R_\C$ that consists of polynomials whose coefficients are in $\mathcal V$, and we let $\cD_{\mathcal V}$ denote the ring of $\mathcal V$-linear differential operators on $R_{\mathcal V}$. Note that there is a natural inclusion $\cD_{\mathcal V} \sq \cD_\C$.

We choose such a finitely generated $\Z$-subalgebra $\mathcal V \sq \C$ such that $f \in R_{\mathcal V}$,  that $b_f(s) \in \mathcal V[s]$, and  that $P(s) \in \cD_{\mathcal V}[s]$. We furthermore assume that $\mathcal V$ is smooth over $\Z$.

Specializing Equation (\ref{eqn-bfunction}) to any integer $a \in \Z$ gives
\begin{equation} \label{eqn-sp-bfunction}
	b_f(a) f^a = P(a) \cdot f^{a+1}, 
\end{equation}
and, by our choice of $\mathcal V$, this equation holds in $R_{\mathcal V}[f^{-1}]$.

Choose a maximal ideal $\fm \sq \mathcal V$ with $f \notin \fm R_{\mathcal V}$. Recall that $\mathcal V / \fm$ is a finite field; we let $p$ denote its characteristic. Given a positive integer $m \geq 0$, we let $V_m \ceq \mathcal V / \fm^{m+1}$, let $R_m \ceq V_m [x_1, \dots , x_n]$, and let $f_m \ceq [f \mod \fm^{m+1} R_{\mathcal V}] \in R_m$. Since $\mathcal V$ was assumed to be smooth over $A$, $V_m$ has a lift of Frobenius $F\colon V_m \to V_m$, which we assume to be surjective, and therefore $V_m$ satisfies the conditions of Setup~\ref{setup-V}. Note that $R_m$ has a compatible lift of Frobenius, which by an abuse of notation we also denote by $F\colon R_m \to R_m$.
 
We let $\cD_m$ denote the ring of $V_m$-linear differential operators on $R_m$; note that there is a natural surjective map $\cD_{\mathcal V} \to \cD_m$ which identifies $\cD_m \cong \cD_{\mathcal V} / \fm^{m+1} \cD_{\mathcal V}$. Given a differential operator $Q \in \cD_{\mathcal V}$ (or, more generally, $Q \in \mathcal V_\fm \otimes_V \mathcal D_{\mathcal V}$), we let $Q_m \ceq [Q \mod \fm^{m+1} \cD_{\mathcal V}] \in \cD_m$ denote its image under this map. Recall that $\cD_m$ comes with a level filtration induced by $F$:
$$\cD_m = \bigcup_{e = 0}^\infty \cD_m^{(F, e)}.$$

If $\alpha \in \Z_p$ is a $p$-adic integer, we can consider the strength $\str(\alpha, f_m)$ of $\alpha$ in $f_m$; varying $m$, we obtain a sequence
$$\left(\str(\alpha, f_m) \right)_{m = 0}^\infty,$$
which is nondecreasing by Proposition~\ref{prop-strength-red}. 

Note that if $\alpha \in \Z_{(p)}$ is a rational number whose denominator is not divisible by $p$, then the rational number $b_f(\alpha)$ lies in the localization $\mathcal V_{\fm} \sq \C$ and the differential operator $P(\alpha)$ lies in the localization $\mathcal V_\fm \otimes_{\mathcal V} \cD_{\mathcal V} \sq \cD_\C$. In particular, it makes sense to consider the image of $b_f(\alpha)$ in $V_m$ and the image $P(\alpha)_m$ of $P(\alpha)$ in $\cD_m$. Note that, if $\beta \in \Z_{(p)}$ is another such rational number with $\alpha \equiv \beta \mod p^{m+1} \Z_{(p)}$, in $V_m$ we have $b_f(\alpha) = b_f(\beta)$, and in $\cD_m$ we have $P(\alpha)_m = P(\beta)_m$.

\begin{theorem}
	Fix notation as above, let $\alpha \in \Z_{(p)}$ be a rational number whose denominator is not divisible by $p$, and let $m \geq 0$ be an integer. In the $\cD_m$-module $R_m [f_m^{-1}] \bs{f_m^\alpha}$, we have
	$$b_f(\alpha) \bs{f_m^\alpha} \in \cD_m \cdot f_m \bs{f_m^\alpha}.$$
\end{theorem}
\begin{proof}
	Pick an integer $a \in \Z$ for which $a \equiv \alpha \mod p^{m+1} \Z_p$ (for example, by truncating the $p$-adic expansion of $\alpha$). By Equation (\ref{eqn-sp-bfunction}), over $V_m$ we obtain
	\begin{align*}
		b_f(\alpha) f_m^a & = b_f(a) f_m^a = P(a)_m \cdot f_m^{a+1} = P(\alpha)_m \cdot f_m^{a+1}.
	\end{align*}
	Choose $e \gg 0$ large enough so that $P(\alpha)_m \in \cD_m^{(F, e)}$, and, replacing $a$ if necessary, assume that $a$ satisfies  $a \equiv \alpha \mod p^{e + m} \Z_p$. We then get (see Section~\ref{subscn-Rf-alpha})
	\begin{align*}
		P(\alpha)_m \cdot f_m \bs{f_m^\alpha} & = f_m^{-a} \left(P(a)_m \cdot f_m^{a+1}\right) \bs{f_m^\alpha} \\
			& = f^{-a}_m\left( b_f(\alpha) f^a_m \right) \bs{f^\alpha_m} \\
			& = b_f(\alpha) \bs{f^\alpha_m}. \qedhere
	\end{align*}
\end{proof}

\begin{corollary} \label{cor-strength-b}
	We have
	$$\nu_p \left( b_f(\alpha) \right) \geq \str(\alpha, f_m),$$
	where $\nu_p(-)$ denotes $p$-adic valuation. In particular, whenever the sequence $( \str(\alpha, f_m) )_{m = 0}^\infty $ is unbounded, we have $b_f(\alpha) = 0$.
\end{corollary}


\newcommand{\etalchar}[1]{$^{#1}$}

\end{document}